\documentclass[11pt]{amsart}

 \usepackage[utf8]{inputenc}
 \usepackage[T1]{fontenc}

 \usepackage[english]{babel}

 \usepackage{amsmath,amsfonts,amssymb}
 \usepackage{mathrsfs}
 \usepackage{dsfont} 
 \usepackage{mathtools} 
 \usepackage{esint} 
 \usepackage{xfrac} 
 \usepackage{defs} 
 \usepackage{fonts} 

 \usepackage{amsthm}
 \usepackage[shortlabels]{enumitem} 
 \usepackage[colorlinks=true,urlcolor=blue, citecolor=red,linkcolor=blue,linktocpage,pdfpagelabels, bookmarksnumbered,bookmarksopen]{hyperref} 
 \usepackage[paper=a4paper,hmargin=1in,vmargin=1in, marginparwidth=0.8in]{geometry} 

 \usepackage[colorinlistoftodos,prependcaption,linecolor=blue,backgroundcolor=blue!25,bordercolor=blue,textsize=tiny]{todonotes}
 \usepackage{natbib}
 \usepackage{cleveref}
 \usepackage{comment}

 \usepackage{tikz}
 \usetikzlibrary{arrows.meta,calc,decorations.pathreplacing}

plus 6pt minus 12pt
\numberwithin{equation}{section}

\newtheorem{theorem}{Theorem}[section]
\newtheorem{proposition}[theorem]{Proposition}
\newtheorem{lemma}[theorem]{Lemma}
\newtheorem{corollary}[theorem]{Corollary}
\theoremstyle{remark}
\newtheorem{remark}[theorem]{Remark}

\theoremstyle{definition}

\crefname{theorem}{Theorem}{Theorems}
\crefname{proposition}{Proposition}{Propositions}
\crefname{lemma}{Lemma}{Lemmas}
\crefname{corollary}{Corollary}{Corollaries}

\title[Nonlocal geometric rigidity inequalities]{Nonlocal geometric rigidity inequalities and the justification of linearized peridynamics}
 \author{James M. Scott}
 \address[James M. Scott]{Department of Mathematics and Statistics,
 Auburn University,
 305 W Samford Avenue,
Auburn University, AL 36849, USA}

\begin{document}

\begin{abstract}

We prove geometric rigidity inequalities for vector fields using double-integral Gagliardo-type quantities. 
We show that the nonlocal-norm distance of the difference quotient of a vector field from the difference quotient of a single rigid motion is controlled by the same nonlocal norm of a difference characterizing the distance to the manifold of all rigid motions.
The inequality is then used to show the variational convergence of a sequence of rescaled nonlocal fully nonconvex integral functionals to a linearized energy functional. 
Related functionals arise in peridynamic formulations of continuum mechanics via the total strain energy; this variational convergence provides a rigorous justification of linearized peridynamics with volume constraint via the convergence of rescaled displacements.
We additionally establish analogous rigidity inequalities in the linearized setting, which can be thought of as nonlocal analogues of Korn's inequality for displacements.
\end{abstract}

\maketitle

\section{Introduction}
For any domain $\Omega \subset \bbR^n$ and any vector field $v : \Omega \to \bbR^n$
\begin{equation}\label{eq:Intro:Isometry}
	|v(x)-v(y)| = |x-y| \quad \forall x,y \in \Omega \qquad \text{ if and only if } \qquad v(x) = Ax + b \text{ in } \Omega,
\end{equation}
for a constant $n \times n$ matrix $A$ and constant vector $b \in \bbR^n$. The matrix $A$ is a rotation, i.e. it satisfies $A^T A = I$, where $A^T$ is the transpose of $A$ and $I$ is the identity matrix. The set of all such maps $Ax+b$ with $A^T A = I$ are the \textit{rigid motions} of $\bbR^n$, a class denoted here by $\scR$.
The quantity 
\begin{equation*}
	\rme[v](x,y) := \frac{|v(x)-v(y)|}{|x-y|} - 1
\end{equation*}
is therefore a measure of how $v$ stretches or shrinks the vector $x-y$. If $\rme[v](x,y) = 0$, then the lengths $x-y$ are unchanged by the deformation of $\Omega$ under the map $v$.

In this work we show a quantitative analogue of \eqref{eq:Intro:Isometry}; that is, if $e[v]$ is close to $0$ then $v$ is nearly a rigid motion. The notions ``close to'' and ``nearly'' are characterized via Lebesgue spaces $L^p(\Omega;\bbR^n)$ and nonlocal Sobolev-Slobodeckij-type spaces.
For $1 \leq p < \infty$ and a nonnegative weight function $\gamma$ satisfying certain growth conditions, consider the quantity
\begin{equation}\label{eq:Intro:Energy}
	E(v) = \int_{\Omega} \int_{\Omega} \gamma(x-y) |\rme[v](x,y)|^p \, \rmd y \, \rmd x.
\end{equation} 
We prove the following inequality: there exists a universal constant $C$ such that for all vector fields $v$ with $\rme[v]$ belonging to a weighted subspace of $L^p(\Omega \times \Omega)$
\begin{equation}\label{eq:Intro:GeometricRigidity}
	\inf_{\psi \in \scR} \int_\Omega \int_{\Omega} \gamma(x-y) \left| \frac{v(x) - v(y)}{|x-y|} - \frac{\psi(x) - \psi(y)}{|x-y|} \right|^p  \, \rmd y \, \rmd x \leq C E(v).
\end{equation}
The result is stated precisely in \cref{thm:Intro:GeometricRigidity} below.
The quantity on the left-hand side is the Gagliardo seminorm characterizing a Banach space.
However, the right-hand side 
is not a seminorm. In fact, the set of measurable functions for which the right-hand side vanishes is precisely $\scR$ \cite{Mengesha2024Linearization}.

Quantities similar to \eqref{eq:Intro:Energy} are used to measure elastic material response in a nonlocal continuum mechanics model known as peridynamics \cite{Silling2000Reformulation,Silling2010Linearized,Silling2007Peridynamic}. In this model, the vector $x-y$ represents the bond between two material points in the elastic body $\Omega$, and likewise $v(x) - v(y)$ represents said bond subjected to a deformation $v : \Omega \to \bbR^n$. The stretch or shrink of the bond is given by the Euclidean norm of the difference quotient 
$$
D[v](x,y) := \frac{v(x)-v(y)}{|x-y|}
$$ 
and the deviation of the bond from elastic equilibrium is given by the \textit{microstrain} $\rme[v](x,y) = |D[v](x,y)| - 1$ and defined above.
Functionals describing the total stored potential energy for elastic homogeneous isotropic bond-based peridynamic materials are written in terms of $\rme[v]$; one archetype is
\begin{equation}\label{eq:Intro:PeriEnergy}
	\int_\Omega \int_\Omega \gamma(x-y) W(\rme[v](x,y)) \, \rmd y \, \rmd x,
\end{equation}
where $W : [-1,\infty) \to [0,\infty)$ is a pairwise microelastic potential. The strain energy density contains the nonnegative integrable weight $\gamma$, also known as the kernel function, which characterizes the strength and extent of material bonds $x-y$.
We assume that $W$ satisfies conditions natural to physical stress-free models, namely
\begin{equation}\label{eq:AssumptionsOnW:1}
	W(r) \geq 0, \text{ and } W(r) = 0 \text{ if and only if } r = 0. 
\end{equation}
Therefore the total strain energy is $0$ if and only if the deformation $v$ belongs to $\scR$. Therefore, we expect that, for small total strain energy (even if the strain itself is unbounded), the deformation is close to a rigid motion; results such as \eqref{eq:Intro:GeometricRigidity} help quantify such statements precisely.

Central to the proof of \eqref{eq:Intro:GeometricRigidity} / \Cref{thm:Intro:GeometricRigidity} is a geometric statement: suppose that a collection of points ``nearly'' forms an orthonormal frame and suppose their mutual distances are changed by the deformation $v$ on the order of $|\rme[v]|$. Then there exists a rigid motion that is close to the deformation within a neighborhood of the origin (See \cite[Proposition 3.13]{Kohn1982New} for the precise statement, reproduced in \Cref{lma:GeometricLemma} below). This ``on average'' result is then integrated to obtain an estimate on Euclidean balls, which are then used to localize estimates on $\Omega$. The localization incorporates all scales simultaneously via summation, and fortunately the singularity conditions on the nonlocal kernel $\gamma$ ensures the weighted summation that appears after incorporating $\gamma$ is finite.

\Cref{thm:Intro:GeometricRigidity} is then applied to obtain the convergence of minimizers of rescalings of the energies \eqref{eq:Intro:PeriEnergy} to a linearized (i.e. convex) energy. A common feature of elastic energy functionals such as \eqref{eq:Intro:PeriEnergy} is their global nonconvexity in the deformation. That is, $v \mapsto W(\rme[v])$ is nonconvex as a function of $|D[v]|$, even if $W$ is convex. The difficulty inherent in working with nonconvex energies makes linearizing the models highly desirable in practice, but challenging to justify in physical contexts.

To illustrate such challenges, consider the following formal derivation of the linearized peridynamics model:
It is expected that small external loads result in small displacements. 
We write the deformation as $v(x) = x + u(x)$, where $u$ is the displacement field; $u = 0$ corresponds to the undeformed configuration, with corresponding deformation $\mathrm{id}(x) = x + 0 = x$.
We consider a loading field $\ell(x;\veps)$ smoothly varying in the (adimensional) parameter $\veps > 0$ with the property that $\ell(x;\veps) = \veps \ell(x) + O(\veps^2)$ for a fixed external loading force $\ell : \Omega \to \bbR^n$.
The formal expectation is that the resulting displacement field $u(x;\veps)$ behaves similarly: $u(x;\veps) = \veps u(x) + O(\veps^2)$ for a fixed vector field $u$.
The corresponding deformation is $v(x;\veps) = x + \veps u + O(\veps^2),$ and the first-order approximation of the associated microstrain is $\rme [ v(\cdot;\veps) ](x,y) = \left| D [\mathrm{id}](x,y) + \veps D[u](x,y) \right| -1 $. 
The equilibrium configurations are then stationary points of the total peridynamic elastic energy 
\begin{equation}\label{eq:Intro:SmallDisplacementEnergy}
	\int_\Omega \int_\Omega \gamma(x-y) W(\rme[\mathrm{id}+\veps u](x,y)) \, \rmd y \, \rmd x - \veps^2 \int_\Omega \ell(x) \cdot u(x) \, \rmd x.
\end{equation}
Writing $|D[\mathrm{id}](x,y) + \veps D[u](x,y)| -1
= \frac{|D[\mathrm{id}](x,y) + \veps D[u](x,y)|^2 -1}{|D[\mathrm{id}](x,y) + \veps D[u](x,y)| + 1}$ and expanding, 
\begin{equation}\label{eq:LinearizedStrain}
	\begin{split}
		&\rme[\mathrm{id} + \veps u](x,y) \\
		=& \veps D[u](x,y) \cdot D[\mathrm{id}](x,y) \\
		&\quad + \frac{\veps^2 |D[u](x,y)|^2 + \veps D[u](x,y) \cdot D[\mathrm{id}](x,y)(1 -|D[\mathrm{id}](x,y) + \veps D[u](x,y)|) }{|D[\mathrm{id}](x,y) + \veps D[u](x,y)| + 1}\\
		=& \veps D[u](x,y) \cdot D[\mathrm{id}](x,y) + O(\veps^{2} |D[u](x,y)|^2)\,.
	\end{split}
\end{equation}
We assume that
\begin{equation}\label{eq:AssumptionsOnW:2}
	\begin{gathered}
		\text{There exists a neighborhood of } r=0 
		\text{ on which } W(r) \text{ is } C^2, \text{ with }
		W''(0) := \mu > 0.
	\end{gathered}
\end{equation}
Now, if the vector field $u$ is assumed to be Lipschitz, then Taylor expanding $W$ and using \eqref{eq:AssumptionsOnW:1}, \eqref{eq:AssumptionsOnW:2} and \eqref{eq:LinearizedStrain} gives
\begin{equation}\label{eq:LinearizedStrainPotential}
	W(\rme[\mathrm{id}+\veps u](x,y)) = \veps^2 \frac{\mu}{2} (D[u](x,y) \cdot D[\mathrm{id}](x,y))^{2} + o(\veps^2)\,.
\end{equation}
Then dividing the total energy in \eqref{eq:Intro:SmallDisplacementEnergy} by $\veps^2$ and taking $\veps \to 0$, we obtain the functional 
\begin{equation}\label{eq:Intro:LinearizedEnergy}
	\begin{split}
		\frac{\mu}{2} \int_\Omega \int_\Omega \gamma(x-y) ( D[u](x,y) \cdot D[i](x,y) )^2 \, \rmd y \, \rmd x - \int_\Omega \ell \cdot u \, \rmd x.
	\end{split}
\end{equation}
This expression is usually taken as the total energy in linearized bond-based peridynamics models.  
The quantity $D[u](x,y) \cdot D[\mathrm{id}](x,y)$ can thus be viewed as a nonlocal analogue of linearized microstrain.

The convergence of the re-scaled sequence of potential energies is under the assumption that the field $u$ is Lipschitz, which excludes vector fields that may have discontinuities or large strain $\rme[v]$. In \cite[page 222]{Silling2010Crack}, it is argued that weakening the condition for linearization is advantageous because it ``permits us to consider perturbations containing jump discontinuities in ${u}$.'' Indeed, the ability to model discontinuous deformation fields is a central attractive feature of peridynamics models.
In this regard, our first main result justifies the linearized bond-based peridynamics model for vector fields that make \eqref{eq:Intro:SmallDisplacementEnergy} finite, which is a class of vector fields more general than Lipschitz. In particular, the microstrain $\rme[v]$ is allowed to be unbounded. Moreover, depending on the singularity of $\gamma$, this class may contain discontinuous vector fields. We will state the result precisely in the next section but broadly speaking we will prove that minimizers (subject to some volumetric conditions) of the functionals with rescaled kernel $\rho$ 
\begin{equation}\label{nl-total-energy}
	\frac{1}{\veps^2 } \int_\Omega \int_\Omega |x-y|^2 \rho(x-y) W(\rme[\mathrm{id}+\veps u](x,y)) \, \rmd y \, \rmd x - \int_\Omega \ell(x) \cdot u(x) \, \rmd x
\end{equation}
converge in the appropriate topology to the minimizer of
\begin{equation}\label{quad-energy}
	\frac{\mu}{2} \int_\Omega \int_\Omega |x-y|^2 \rho(x-y) ( D[u](x,y) \cdot D[\mathrm{id}](x,y) )^2 \, \rmd y \, \rmd x - \int_\Omega \ell(x) \cdot u(x) \, \rmd x.
\end{equation}

We use the framework of $\Gamma$-convergence.
A similar $\Gamma$-convergence result was obtained by T. Mengesha and the author in \cite{Mengesha2024Linearization}. However, compactness of a finite-energy sequence, and thus the convergence of minimizers, remained out of reach. Here, we provide a complete story; the minimizers of the nonlinear energy converge to the minimizer of the linearized energy. 
The $\Gamma$-limit is calculated in this work with the presence of constraints not present in \cite{Mengesha2024Linearization}, as well as distinct assumptions on $W$, so the corresponding proofs take a different route.

\section{Setup and main results}
We denote the Lebesgue measure in dimension $n$ by $\cL^n$. All results of the paper hold for $n \geq 1$ unless indicated otherwise. Euclidean balls are denoted $B(x_0,r)$, with center $x_0$ and/or radius $r$ suppressed if clear from context. For a measurable set $A$, $\mathds{1}_A$ denotes its indicator function, which is $1$ on $A$ and $0$ otherwise.
The set of $n \times n$ matrices is associated with $\bbR^{n \times n}$. We denote the unit sphere in $n$ dimensions as $\bbS^{n-1} \subset \bbR^n$. 
The reference configuration $\Omega \subset \bbR^n$ is assumed to be a bounded Lipschitz domain. Lebesgue and fractional/classical Sobolev spaces are denoted $L^p$ and $W^{s,p}$ respectively.

The nonlocal Banach space is defined using the kernel function $\rho$ as
\begin{equation*}
	\cS^{\rho,p}(\Omega;\bbR^n) := \left\{ v \in L^p(\Omega;\bbR^n) \, : \, [v]_{\cS^{\rho,p}(\Omega)}^p := \int_\Omega \int_\Omega \rho(x-y) |v(x)-v(y)|^p \, \rmd y \, \rmd x < \infty  \right\},
\end{equation*}
with norm $\vnorm{u}_{\cS^{\rho,p}(\Omega)}^p := \vnorm{u}_{L^p(\Omega)}^p + [v]_{\cS^{\rho,p}(\Omega)}^p$.
The kernel $\rho$ defining the space $\cS^{\rho,p}$ satisfies the following assumptions:
\begin{equation}\label{assu:Kernel}
	\begin{gathered}
		|z|^p \rho(z) \in L^1_{loc}(\bbR^n), \quad  \rho \geq 0, \quad \rho \text{ is radial, } \quad \text{ and } \rho(r) \text{ is nonincreasing for } r \in (0,\infty).  \tag{K}
	\end{gathered}
\end{equation}

We state our first main result precisely in the following theorem:

\begin{theorem}\label{thm:Intro:GeometricRigidity}
	Let $p \in [1,\infty)$. Suppose $\rho$ satisfies \eqref{assu:Kernel} and the following singularity condition:
	\begin{equation}\label{assu:GeomRigid}
		\text{There exists } r_0 > 0 \text{ and } q \in (0,2^{n+p}) \text{ such that } \rho(r/2) \leq q \rho(r) \quad \forall r \in (0,r_0). \tag{$\rmK_{\rmG \rmR}$}
	\end{equation}
	Then there exists a constant $C>0$ depending only on $n$, $p$, $\rho$, and $\Omega$ such that the following holds: for every $v \in \cS^{\rho,p}(\Omega;\bbR^n)$ there exists $\psi \in \scR$ such that
	\begin{equation}\label{eq:GeometricRigidity}
		\int_\Omega \int_{\Omega} \rho(x-y) |x-y|^p |D[v-\psi](x,y)|^p  \, \rmd y \, \rmd x \leq C \int_{\Omega} \int_{\Omega} \rho(x-y) |x-y|^p |\rme[v](x,y)|^p \, \rmd y \, \rmd x.
	\end{equation}
\end{theorem}

The condition \eqref{assu:GeomRigid} requires that $\rho$ not be too singular at $r = 0$.
The class of kernels satisfying both \eqref{assu:Kernel} and \eqref{assu:GeomRigid} include $\rho(z) = |z|^{-n-sp}$ for any $p \in [1,\infty)$ and any $s \in (0,1)$, and $\rho(z) = \eta^{-n-p} \mathds{1}_{ \{|z| < \eta \} }$ for any $\eta > 0$.
The requirement on the constant $q$ is sharp in the sense that there exists $\rho$ satisfying \eqref{assu:Kernel} with $\rho(r/2) < 2^{n+p} \rho(r)$ for $r \in (0,1)$ such that \eqref{eq:GeometricRigidity} fails. This will be shown in \cref{subsec:Counterexamples}.

We will consider minimizers of \eqref{eq:Intro:SmallDisplacementEnergy} subject to a volume constraint, which is a nonlocal analogue of Dirichlet boundary conditions. Given another Lipschitz domain $\omega \subset \Omega$ such that $\Omega \setminus \overline{\omega}$ remains a Lipschitz domain, and for any $g \in \cS^{\rho,p}(\Omega;\bbR^n)$ we introduce the set
\begin{equation*}
	\cS^{\rho,p}_g := \{ u \in \cS^{\rho,p}(\Omega;\bbR^n) \, : \, u = g \text{ a.e. on } \omega \}.
\end{equation*}
Throughout this work, $W : [-1,\infty) \to [0,\infty)$ is taken to be an $\cL$-measurable function satisfying \eqref{eq:AssumptionsOnW:1}, \eqref{eq:AssumptionsOnW:2} and
\begin{equation}\label{eq:AssumptionsOnW:3}
	\text{There exists a constant } c_0 > 0 \text{ such that } W(r) \geq c_0 r^2 \text{ for all } r \in [-1,\infty).
\end{equation}
Note that $W$ may have more than quadratic growth in general.

The external load is modeled by a continuous linear functional $\scL : L^2(\Omega;\bbR^n) \to \bbR$, i.e. $\scL(v) = \int_\Omega \ell \cdot v$ for some $\ell \in L^2(\Omega;\bbR^n)$. Then the deformation $v$ corresponding to the load $\veps \scL$ has the form $v(x) = x + \veps u(x)$ and minimizes the rescaled energy
\begin{equation}\label{eq:NonlinearEnergy}
	\cF_\veps(u) := \frac{1}{\veps^2} \int_\Omega \int_\Omega \rho(x-y) |x-y|^2 W(\rme[\mathrm{id}+\veps u](x,y)) \, \rmd y \, \rmd x - \scL(u).
\end{equation}
We assume the volume constraint condition on the displacement $u$ so that $v(x) = x + \veps g(x)$ almost everywhere on $\omega$.
The following theorem describes the behavior of minimizers as $\veps \to 0$, and is the second main result of the paper.
\begin{theorem}\label{thm:VariationalConvergence}
	Let $W$ be as above, and let $g \in W^{1,\infty}(\Omega;\bbR^n)$. Let $\rho$ satisfy \eqref{assu:Kernel}, \eqref{assu:GeomRigid}, and the condition \eqref{assu:Compact} specified below. For $\veps > 0$ define
	\begin{equation*}
		m_\veps := \inf_{u \in \cS^{\rho,2}_g} \cF_\veps(u)
	\end{equation*}
	and let $\{ u_\veps \}_\veps \subset \cS^{\rho,2}_g$ be a sequence satisfying
	\begin{equation}\label{eq:AlmostMins}
		\cF_\veps(u_\veps) = m_\veps + o(1).
	\end{equation}
	Then $\{ u_\veps \}$ converges strongly in $L^2(\Omega;\bbR^n)$ to the unique solution of 
	\begin{equation*}
		m := \min_{ u \in \cS^{\rho,2}_g } \cF(u),
	\end{equation*}
	where 
	\begin{equation}\label{eq:LinearizedEnergy}
		\cF(u) := \frac{\mu}{2} \int_\Omega \int_\Omega \rho(x-y) |x-y|^2 (D[u](x,y) \cdot D[\mathrm{id}](x,y))^2 \, \rmd y \, \rmd x - \scL(u).
	\end{equation}
	Moreover, $m_\veps \to m$.
\end{theorem}

\cref{thm:VariationalConvergence} establishes the convergence of ``almost minimizers'' to a minimizer of the linearized problem. The proof will follow from sufficient criteria in the paradigm of $\Gamma$-convergence. These sufficient criteria involve the nonlocal stress functionals
\begin{equation*}
	\cE_\veps(u) := \cF_\veps(u) + \scL(u), \qquad \cE(u) := \cF(u) + \scL(u).
\end{equation*}

\begin{theorem}\label{thm:EnergyBound}
	For $W$ as in \cref{thm:VariationalConvergence} and $\rho$ satisfying \eqref{assu:Kernel} and \eqref{assu:GeomRigid}, there exists a constant $C = C(n,\rho,\Omega,\omega) > 0$ such that for every $g \in \cS^{\rho,2}(\omega;\bbR^n)$ and every sequence $\{ u_\veps \} \subset \cS^{\rho,2}_g$ we have
	\begin{equation*}
		\vnorm{ u_\veps }_{\cS^{\rho,2}(\Omega)}^2 \leq C \left(\cE_\veps(u_\veps) + \vnorm{ g }_{\cS^{\rho,2}(\omega)}^2  \right).
	\end{equation*}
\end{theorem}
This theorem ensures that, for any sequence $\{u_\veps\} \subset \cS^{\rho,2}_g$ such that $\{ \cE_\veps(u_\veps) \}$ is bounded, the sequence is also bounded in the $\cS^{\rho,2}$ norm. Since $\cS^{\rho,2}$ is reflexive, it follows that $\{ u_\veps \}$ has a subsequence that is weakly convergent in $\cS^{\rho,2}$. 
We then invoke a compact embedding result from \cite{Du2023$}, which will imply the strong $L^2$ convergence of the subsequence so long as the kernel $\rho$ satisfies the singularity condition
\begin{equation}\label{assu:Compact}
	\lim\limits_{\delta \to 0} \delta^p \left( \int_{B(0,\delta)} |z|^p \rho(z) \, \rmd z \right)^{-1} = 0 \tag{$\rmK_{\rmC}$}
\end{equation}
for $p = 2$. Further discussion on the conditions \eqref{assu:GeomRigid} and \eqref{assu:Compact} is contained in \cref{subsec:Counterexamples}.

\begin{theorem}\label{thm:GammaConvergence}
	Under the assumptions of \cref{thm:VariationalConvergence}, define the functional $\overline{\cE}_\veps : L^2(\Omega;\bbR^n) \to \overline{\bbR}$ via
	\begin{equation*}
		\overline{\cE}_\veps(u) :=
		\begin{cases}
			\cE_\veps(u), &\quad \text{if } u \in \cS^{\rho,2}_g, \\
			+\infty, &\quad \text{if } L^2(\Omega;\bbR^n) \setminus \cS^{\rho,2}_g,
		\end{cases}
	\end{equation*}
	and define $\overline{\cE} : L^2(\Omega;\bbR^n) \to \overline{\bbR}$ in the same way:
	\begin{equation*}
		\overline{\cE}(u) :=
		\begin{cases}
			\cE(u), &\quad \text{if } u \in \cS^{\rho,2}_g, \\
			+\infty, &\quad \text{if } L^2(\Omega;\bbR^n) \setminus \cS^{\rho,2}_g.
		\end{cases}
	\end{equation*}
	Then the total strain energies $\overline{\cF}_\veps = \overline{\cE}_\veps - \scL$ and  $\overline{\cF} = \overline{\cE} - \scL$ satisfy, for any sequence $\veps \to 0^+$,
	\begin{equation*}
		\Gammalim_{\veps \to 0^+} \overline{\cF}_{\veps}(u) = \overline{\cF}(u)
	\end{equation*}
	in the strong topology on $L^2(\Omega;\bbR^n)$.
\end{theorem}
\cref{thm:GammaConvergence}, together with the compactness implied by \cref{thm:EnergyBound} and \eqref{assu:Compact}, ensures the convergence of minimizers. \cref{thm:VariationalConvergence} concerns the behavior of almost-minimizers for very general energies; if we assume that $W$ is additionally lower semicontinuous, then we can actually show via direct methods of the calculus of variations that the functional $\cF_{\veps}$ has a minimizer in $\cS^{\rho,2}_{g}$ for each $\veps$; see \cref{subsec:Existence} for the proof. Uniqueness of the minimizer for $\cF$ follows from the strict convexity of $\cE$.

\begin{remark}
	In the absence of the volume constraint, \cref{thm:EnergyBound} is false, as the following counterexample shows. Suppose $\omega = \emptyset$. Then let $W(r) = r^2$ and let $u_\veps(x) = \frac{2x}{\veps}$. Then $\vnorm{u_\veps}_{\cS^{\rho,2}(\Omega)}^2 \geq C(\Omega,\rho)\veps^{-2} \to \infty$ as $\veps \to 0$, while $\cE_\veps(u_\veps) = 0$ for all $\veps > 0$.
\end{remark}

\begin{remark}
	The condition that $g \in W^{1,\infty}(\Omega;\bbR^n)$ cannot be replaced with any weaker condition in general, unless appropriate upper bounds are satisfied on $W$. To illustrate this, consider the setting $\rho(z) = |z|^{-n-2s}$ for some $s \in (0,1)$, so that $\cS^{\rho,2}(\Omega;\bbR^n)$ coincides with the fractional Sobolev space $W^{s,2}(\Omega;\bbR^n)$. 
	Then the $\Gamma$-convergence result fails in general if $g \in W^{\alpha,2}(\Omega;\bbR^n)$ for any fixed $\alpha \in [s,1)$. Indeed, suppose that $W$ satisfies \eqref{eq:AssumptionsOnW:1}, \eqref{eq:AssumptionsOnW:2} and \eqref{eq:AssumptionsOnW:3}, with the additional condition that there is a $q > 2 \frac{1-s}{1-\alpha}$ such that 
	$$
	W(r) \geq C |r|^q, \qquad \text{ for sufficiently large } |r|.
	$$
	Let $g \in W^{\alpha,2}(\Omega;\bbR^n)$ with $g \notin W^{\alpha,q}(\omega;\bbR^n)$; such a choice is possible since $q > 2$ and since $\omega$ has nonzero $\cL^n$-measure. 
	Now we claim that 
	$$
	\overline{\cF}_\veps(u) = \infty \text{ for all } u \in W^{s,2}_g \text{ and all } \veps > 0.
	$$
	Suppose to the contrary, that $\overline{\cF}_\veps(u) < \infty$ for some $u \in W^{s,2}_g$. Then by the lower bound on $W$ assumed above, $\mathrm{id} + \veps u \in W^{1-\frac{2}{q}(1-s),q}(\Omega;\bbR^n)$, hence $u \in W^{1-\frac{2}{q}(1-s),q}(\Omega;\bbR^n) \subset W^{\alpha,q}(\Omega;\bbR^n)$ by assumption on $q$. But $[g]_{W^{\alpha,q}(\omega)} = [u]_{W^{\alpha,q}(\omega)} \leq [u]_{W^{\alpha,q}(\Omega)} < \infty$, a contradiction and the claim is proved. However, $\overline{\cF}(g) < \infty$ since $g \in W^{\alpha,2}(\Omega;\bbR^n) \subset W^{s,2}(\Omega;\bbR^n)$. Thus $\overline{\cF}_\veps(u)$ cannot possibly $\Gamma$-converge to $\overline{\cF}(u)$.
\end{remark}

We next remark on some straightforward generalizations of our setting. First, energies $\cE_\veps$ defined over more general nonlocal measures of strain can be considered. For example, define the nonlocal pairwise microstrain inspired by Seth-Hill material models via
\begin{equation*}
	\rme_m[v](x,y) = \frac{1}{m} ( |D[v](x,y)|^m - 1 ),
\end{equation*}
for a fixed positive constant $m \geq 1$. Then \cref{thm:Intro:GeometricRigidity} remains true, since $|\rme[v]| \leq m |\rme_m[v]|$ for all $x$ and $y$. Moreover, the linearized functional $\cE$ remains of the same form, although the constant $\mu$ may change. As a result, a version of \cref{thm:VariationalConvergence} can be proved; the extended functionals $\overline{\cE}_\veps$ may in general be defined differently due to the $m$-dependent growth rate of $W(\rme_m[v])$.

Peridynamic formulations for heterogeneous materials can also be considered via more general microelastic potentials $W(x,x-y,\rme[v](x,y))$.
For instance, Lipschitz continuous potentials of the form $W(x-y,\rme[v](x,y))$ were considered in \cite{Mengesha2024Linearization}, with $\Gamma$-convergence of the linearization obtained under a different set of assumptions.
In fact, in the case that $W$ depends only on $\rme[v]$, the $\Gamma$-convergence result of this work generalizes that of \cite{Mengesha2024Linearization}, since we do not assume any kind of continuity of $W$ far away from $0$.
By appropriately modifying the conditions \eqref{eq:AssumptionsOnW:1}, \eqref{eq:AssumptionsOnW:2}, and \eqref{eq:AssumptionsOnW:3} to be ``uniform'' in $x$, variational convergence can also be proved for energies characterized by $W = W(x,x-y,\rme[v])$ with only slight changes to the proofs here. For the sake of illustrating the main ideas of the convergence we treat models for homogeneous materials only.

We can also apply a similar linearization argument directly to the estimate \eqref{eq:GeometricRigidity} to obtain a linearized version of
\cref{thm:Intro:GeometricRigidity}. Similar results have been obtained across different contexts  \cite{Acosta2026Fractional,Harutyunyan2023Fractional,MengeshaConsistency,Scott2019fractional}, and are known as \textit{nonlocal Korn's inequalities}. Just as \cref{thm:Intro:GeometricRigidity} is a quantitative estimate for the nonlocal distance of a vector field $v$ from the manifold $\scR$ in terms of the quantity $\rme[v]$, nonlocal Korn's inequalities estimate the distance of $v$ from the subspace of affine vector fields $\{ Ax + b \, : \, A = - A^T \}$.

\begin{theorem}\label{thm:Korn}
	Under the assumptions of \cref{thm:Intro:GeometricRigidity}, there exists $C = C(n,p,\rho,\Omega) > 0$ such that the following holds: for every $u \in \cS^{\rho,p}(\Omega;\bbR^n)$ there exists an affine vector field $\psi(x) = Ax + b$ with $A = - A^T$ such that
	\begin{equation}\label{eq:Korn}
		\vnorm{u - \psi}_{\cS^{\rho,p}(\Omega)}^p \leq C \int_\Omega \int_\Omega \rho(x-y) |x-y|^p \big| D[u](x,y) \cdot D[\mathrm{id}](x,y)\big|^p \, \rmd y \, \rmd x.
	\end{equation}
\end{theorem}

Prior to this work inequalities of the form \eqref{eq:Korn} for nonlocal Sobolev-Slobodeckij spaces have been established only in the case of fractional kernels $\rho(z) = |z|^{-n-sp}$ for $s \in (0,1)$. To our knowledge, this is the first result establishing such inequalities for more general kernels.

\subsection{Local and nonlocal hyperelasticity}

Taken together, the assumptions for \cref{thm:Intro:GeometricRigidity} and \cref{thm:VariationalConvergence} are somewhat opposed. In order for both to hold, the kernel $\rho$ must not be too singular, so that \eqref{assu:GeomRigid} holds, while at the same time having enough singularity in order for \eqref{assu:Compact} to hold. 
Therefore the theory described in this paper is ``genuinely nonlocal,'' in that there is no mechanism, formal or rigorous, to recover statements about local objects; this is discussed more below and also in \cref{subsec:Counterexamples}.

On the subject of local rigidity theorems, that is, rigidity statements involving derivatives of vector fields, there have been numerous developments over the past half-century with far-reaching consequences in the field of continuum mechanics. 
In the setting of classical local hyperelasticity, in which elastic energies are characterized by the deformation gradient $\grad v$, a description involving derivatives of smooth unstrained configurations is
\begin{equation*}
	\grad v(x)^T \grad v(x) = I \quad \text{ on } \Omega,
\end{equation*}
Then a statement analogous to \eqref{eq:Intro:Isometry} involving local quantities is the 
Liouville-type theorem
\begin{equation*}
	\begin{gathered}
		\grad v \in \rmO(n) \text{ a.e.}, \\
		\text{and } v, v^{-1} \in W^{1,\infty} 
	\end{gathered}
	\qquad \Rightarrow \qquad v = Ax + b \in \scR,
\end{equation*}
where $\rmO(n) := \{ A \in \bbR^{n \times n} \, : \, A^T A = I \}$ denotes the set of rotations.
This theorem can also be made quantitative in the context of Sobolev spaces; such results were established by Robert V. Kohn in \cite{Kohn1982New}. Among stronger Sobolev embedding-type estimates, he showed that there exists a constant $C > 0$ such that for all $v \in W^{1,\infty}(\Omega;\bbR^n)$ with $v^{-1} \in W^{1,\infty}(v(\Omega);\Omega)$ there exists $\psi \in \scR$ that satisfies
\begin{equation}\label{eq:KohnEstimate}
	\vnorm{u - \psi }_{L^p(\Omega)} \leq C \vnorm{E[\grad v]}_{L^p(\Omega)},
\end{equation}
where $E[\grad v]$ is a function of the eigenvalues of $(\grad v^T \grad v)^{1/2}$ that vanishes only if $\grad v^T \grad v = I$ almost everywhere.
In \cite{John1972Bounds} Fritz John showed stronger Sobolev estimates 
\begin{equation}\label{eq:JohnEstimate}
	\inf_{ \psi \in \scR} \vnorm{ v - \psi }_{W^{1,p}(\Omega)} \leq C(n,p,\Omega) \vnorm{  (\grad v^T \grad v)^{1/2} - I }_{L^p(\Omega)}
\end{equation}
provided that $v$ additionally has uniformly small strain, i.e. for all $C^1$ invertible maps with $\vnorm{ (\grad v^T \grad v)^{1/2} - I }_{L^\infty} < \delta$ for a constructed $\delta(n,p)$. 
Clearly, if \eqref{eq:JohnEstimate} holds then the estimate \eqref{eq:KohnEstimate} follows by an application of the Poincar\'e inequality, but \eqref{eq:KohnEstimate} holds for a larger class of vector fields, without any uniform smallness assumption on the strain.

Both \eqref{eq:KohnEstimate} and \eqref{eq:JohnEstimate} require uniform control over the inverse $v^{-1}$. If this assumption is dropped, then any rigidity estimate will fail in general. Using a convex integration scheme \cite{Mueller2003Convex} one can construct a sequence of Lipschitz vector fields $v_k$ such that
\begin{equation}\label{eq:ConvexInt}
	\vnorm{ (\grad v_k^T \grad v_k)^{1/2} - I }_{L^p(\Omega)} \to 0 \quad \text{ while } \quad \inf_{ \psi \in \scR} \vnorm{ v_k - \psi }_{W^{1,p}(\Omega)} \geq c_0 > 0.
\end{equation}
In fact, this lack of rigidity can be further developed; for any matrix $A$ in the quasiconvex hull of $\rmO(n)$, the affine vector field $Ax$ can be approximated uniformly by a sequence of vector fields $v_k$ which satisfy \eqref{eq:ConvexInt}.
Unlike for \eqref{eq:KohnEstimate} and \eqref{eq:JohnEstimate}, we do not assume any uniform control over the strain $\rme[v]$ or make any invertibility assumption on $v$ at all. Nevertheless, we obtain rigidity in the nonlocal setting; the conclusion of \cref{thm:Intro:GeometricRigidity} is in stark contrast to \eqref{eq:ConvexInt} and does not hold in any localizing regime.

A more fruitful approach to rigidity in local hyperelasticity begins with considering deformations satisfying $\det \grad v > 0$ almost everywhere; not only does this constraint eliminate the unphysical phenomenon of interpenetration of matter, it also restricts unstrained configurations to satisfy $\grad v \in \rmS \rmO(n)$, which denotes the set of rotations with positive determinant. This inclusion is not only rigid, i.e.
\begin{equation*}
	\begin{gathered}
		\grad v \in \rmS \rmO(n) \text{ a.e.}, \\
		\text{and } v \in W^{1,p} 
	\end{gathered}
	\qquad \Rightarrow \qquad v = Ax + b \in \scR \text{ with } A \in \rmS \rmO(n),
\end{equation*}
but it also can be made quantitative. The seminal result of Friesecke, James and M\"uller \cite{Friesecke2002theorem} states that
\begin{equation}\label{eq:FJM}
	\inf_{ \substack{\psi \in \scR \\ \grad \psi \in \rmS \rmO(n)} } \vnorm{ v - \psi }_{W^{1,p}(\Omega)} \leq C(n,p,\Omega) \vnorm{ \dist(\grad v, \rmS \rmO(n)) }_{L^p(\Omega)}, \qquad \forall v \in W^{1,p}(\Omega;\bbR^n).
\end{equation}
Note that no invertibility assumptions are made on $v$ in \eqref{eq:FJM}.
Many hyperelastic energies are coercive with respect to the right-hand side of \eqref{eq:FJM}; the rigidity estimate can be used to obtain the variational convergence of the nonlinear hyperelastic strain energies to the linearized energy \cite{Agostiniani2012Linear,DalMaso2002Linearized}. The approaches of those works to justifying linearization via convergence of minimizers is the approach we take in the present work.

In summary, \cref{thm:Intro:GeometricRigidity} qualitatively resembles \eqref{eq:FJM} for general vector fields in the nonlocal space $\cS^{\rho,2}$, but in fact concerns distances to a different manifold, for which no local analogue holds in Sobolev spaces. 
It is the author's belief that geometric rigidity estimates of the type \eqref{eq:Intro:GeometricRigidity} will be useful for other asymptotic regimes in which nonlocality is preserved, analogously to \eqref{eq:FJM}. The estimate \eqref{eq:FJM} has been used to great effect in studying the behavior of continuum mechanics models in different paradigms, such as dimension reduction in nonlinear elasticity \cite{Friesecke2006hierarchy,Harutyunyan2020Geometric,Lewicka2023Calculus,Mora2004nonlinear}, and  \eqref{eq:Intro:GeometricRigidity} seems to be paramount in studying similar peridynamic formulations. Regarding such questions for other nonlocal models, dimension-reduction regimes for fractional seminorms have been studied in \cite{Braides2026Dimension}, and a linearization result for correspondence-type models with singular kernel has been obtained \cite{Kreisbeck2026Linearization} via reformulation of the model to a generalization of that in \cite{DalMaso2002Linearized}.

This paper is organized as follows: In \cref{sec:GeometricRigidity} we prove the fractional nonlinear geometric rigidity inequality (\cref{thm:Intro:GeometricRigidity}). The proofs of \cref{thm:VariationalConvergence}, \cref{thm:EnergyBound} and \cref{thm:GammaConvergence} are in \cref{sec:Justification}. \Cref{sec:Further} contains discussions on the validity of the results for different nonlocal kernels common in applications, as well as the proof of \cref{thm:Korn}.

\noindent \textbf{Acknowledgement.}\,
The author's research is supported by NSF DMS-2606364. The author also thanks Davit Harutyunyan, Marta Lewicka, Cy Maor, Tadele Mengesha, Hayk Mikayelyan, and Armin Schikorra for helpful and illuminating conversations which influenced the development of this manuscript.

\noindent \textbf{AI Disclosure.}\,
The author had previously obtained the Poincar\'e-type inequality of \cref{thm:GeometricRigidity:Poincare} via a different argument, largely inspired by the technical aspects of \cite{Kohn1982New}, and as a consequence the results of \cref{thm:VariationalConvergence} in the weak-$L^2$ topology. However, while preparing another manuscript \cite{MengeshaConsistency} the author made use of ChatGPT 5.6 Sol and realized that certain outputs  
could be adapted to obtain the stronger geometric rigidity estimate of \cref{thm:Intro:GeometricRigidity}. The only place the author used AI assistance directly in this work is a rewrite of the proof of \cref{lma:KeyElementaryEstimate} and \cref{thm:GeoRigid:Lp:Ball}. Previously, the author had written a much longer covering argument which obscured the nature of the estimate itself.
The weak-$L^2$ convergence of minimizers will be reported later in a different context.
The author takes all responsibility for the content of this manuscript.

\section{Geometric rigidity lemma}\label{sec:GeometricRigidity}
We begin by setting conventions for affine vector fields.
We define for $\psi(x) = Ax+b$ with $A \in \bbR^{n \times n}$ and $b \in \bbR^n$ the following norm
\begin{equation*}
	\Vnorm{\psi}_{\mathrm{aff}} := |A| + |b|,
\end{equation*}
where $|A| = \sqrt{\mathrm{tr}(A A^T)}$ denotes the Frobenius norm.
Elementary estimates and scaling arguments show that for any $n \in \bbN$, $p \in [1,\infty]$, $r > 0$ and $x_0 \in \bbR^n$
\begin{align}
	\frac{1}{C(n,p)} \vnorm{\psi}_{L^p(B(0,1))} &\leq \vnorm{\psi}_{\mathrm{aff}} \leq C(n,p) \vnorm{\psi}_{L^p(B(0,1))}, \label{eq:AffineMapEquivalence:Unscaled} \\
	\frac{1}{C(n,p) r^{n/p} \max\{1,r\} } \vnorm{\psi}_{L^p(B(0,r))} &\leq  \vnorm{\psi}_{\mathrm{aff}} \leq \frac{C(n,p)}{ r^{n/p} \min\{1,r\} } \vnorm{\psi}_{L^p(B(0,r))},  \label{eq:AffineMapEquivalence:Rescaled} \\
		\frac{1}{C(n,p) r^{n/p+1}} \vnorm{\psi}_{L^p(B(x_0,r))} &\leq  |A| \leq \frac{C(n,p)}{r^{n/p+1} } \vnorm{\psi}_{L^p(B(x_0,r))}, \label{eq:AffineMapEquivalence:Rescaled:MatrixPart}
\end{align}
where we take $n/p = 0$ in the case $p = \infty$.
The tensor product of vectors $a$, $b \in \bbR^n$ is denoted $a \otimes b \in \bbR^{n \times n}$.
The next lemma follows from \eqref{eq:AffineMapEquivalence:Rescaled} and a translation argument:

\begin{lemma}
	\label{lma:PolynomialsOnSets}
	Let $B_1 = B(x_1,R_1)$, $B_2 = B(x_2,R_2)$,  $B_3 = B(x_3,R_3)$ be Euclidean balls in $\mathbb{R}^n$ with $B_1, B_2 \subset B_3$. Then for any affine vector field $\psi(x) = Ax + b$ and $p \in [1,\infty)$,
	\begin{equation}\label{eq:AffineNormComp}
		\vnorm{\psi}_{L^p(B_1)} \leq C(n) \left( \frac{R_1}{R_2} \right)^{n/p}
		\frac{R_3}{R_2}
		\vnorm{\psi}_{L^p(B_2)}.
	\end{equation}
\end{lemma}

\subsection{Rigidity for Lebesgue vector fields}

\begin{proposition}\label{prop:Scaled:NoDiff}
	Let $1 \leq p < \infty$ and let $B(x_0,r) \subset \bbR^n$ be a ball. Then there exist constants $M(n) > 1$ and $C(n,p)$ such that the following holds: For any vector field $v \in L^p(B(x_0,Mr);\bbR^n)$ there exists $\psi \in \scR$ such that
	\begin{equation}\label{eq:Scaled:NoDiff}
		\intdm{B(x_0,r)}{|v(x)-\psi(x)|^p}{x} \leq C r^{-n} \iintdm{B(x_0,Mr)}{B(x_0,Mr)}{ \big| |v(x)-v(y)| - |x-y| \big|^p }{y}{x}\,.
	\end{equation}
\end{proposition}
\noindent The geometric fact central to the proof is contained in the following lemma:
\begin{lemma}[\cite{Kohn1982New}, Proposition 3.13]\label{lma:GeometricLemma}
	There exist constants $r_0(n) < 1$ and $C_0(n)$ with the following property:
	Let $\kappa > 0$ and let $\{ \theta_1, \ldots, \theta_n \}$ be an orthonormal basis for $\bbR^n$. Suppose $\{ x_1, \ldots, x_{2n} \} \subset \bbR^n$ and $\{ y_1, \ldots, y_{2n} \} \subset \bbR^n$ satisfy
	\begin{equation*}
		x_i \in B(\kappa \theta_i, \kappa r_0)\,, \qquad x_{i+n} \in B(-\kappa \theta_i, \kappa r_0)\,, \qquad 1 \leq i \leq n\,.
	\end{equation*}
	Then there exists a rigid motion $\psi \in \scR$ such that
	\begin{equation*}
		|y_k - \psi(x_k)| \leq C_0 \max_{1 \leq i,j \leq n} \big| |y_i-y_j | - |x_i-x_j| \big|, \qquad \text{ for } 1 \leq k \leq 2n,
	\end{equation*}
	and for any $z \in B(0,\kappa r_0)$ and any $w \in \bbR^n$
	\begin{equation*}
		|w-\psi(z)| \leq C_0 \left( \max_{1 \leq i \leq n} (|w-y_i| - |z-x_i|)_+ + \max_{1 \leq i,j \leq n} \big| |y_i-y_j | - |x_i-x_j| \big| \right).
	\end{equation*}
\end{lemma}

\begin{proof}[proof of \Cref{prop:Scaled:NoDiff}]
	Set $B = B(0,1)$ and $B' = B(0,1/10)$. Let $r_0$ and $C_0$ be as given in Lemma \ref{lma:GeometricLemma}. Let $\kappa = \frac{1}{10(1+r_0)}$ and let $B'' = B(0,\kappa)$. Let $\{ \theta_1, \ldots, \theta_n \}$ be an orthonormal basis for $\bbR^n$.
	Define the set $A \subset \bbR^{n \times 2n} := \bbR^{2n^2}$ by
	\begin{equation*}
		A := B(\kappa \theta_1, \kappa r_0) \times B(\kappa \theta_2, \kappa r_0) \times \ldots \times B(\kappa \theta_n, \kappa r_0) \times B(- \kappa \theta_1, \kappa r_0) \times \ldots \times B(- \kappa \theta_n, \kappa r_0).
	\end{equation*}
	Write elements of $A$ as $\vec{x} = (x_1, \ldots x_{2n})$. We will apply Lemma \ref{lma:GeometricLemma} to $y_i= v(x_i)$ and $w = v(z)$ for the right choice of $\vec{x} \in A$. 
	Note: $B( \pm \kappa \theta_i, \kappa r_0) \subset B'$ and  $B'' \cap B(\pm \kappa \theta_i, \kappa r_0) = \emptyset$ for all $i$.
	Define for $\lambda > 0$
	\begin{equation*}
		\cM^{\lambda}_{ij} := \left\{ \vec{x} \in A \, : \, |x_i-x_j| \left| \frac{|v(x_i)-v(x_j)|}{|x_i-x_j|} - 1 \right| > \lambda \right\}\,, \quad  1 \leq i,j \leq 2n\,,
	\end{equation*}
	\begin{equation*}
		\cN^{\lambda}_i := \left\{ \vec{x} \in A \, : \, \left( \intdm{B''}{ |x_i-y|^p  \left( \frac{|v(x_i)-v(y)|}{|x_i-y|} - 1 \right)_+^p}{y} \right)^{1/p} > \lambda \right\}\,, \quad 1 \leq i \leq 2n\,,
	\end{equation*}
	and set
	\begin{equation*}
		A^{\lambda} := \bigcup_{i,j = 1}^{2n} \cM^{\lambda}_{ij} \cup \bigcup_{i = 1}^{2n} \cN^{\lambda}_{i}\,.
	\end{equation*}
	We will show that $A \setminus A^{\lambda}$ is nonempty for an appropriate choice of $\lambda$ that depends on the right-hand side of \eqref{eq:Scaled:NoDiff}.
	Using Chebyshev's inequality 
	\begin{equation*}
		\begin{split}
			\cL^{2n^2}(\cM^{\lambda}_{ij}) &\leq \frac{1}{\lambda^p} \iintdm{B}{B}{ |x-y|^p \left| \frac{|v(x)-v(y)|}{|x-y|} - 1 \right|^p }{y}{x} \\ 
			&\leq  \frac{1}{\lambda^p} \iintdm{B}{B}{  |x-y|^p |\rme[v](x,y)|^p  }{y}{x} \qquad \text{ for all } 1 \leq i,j \leq 2n\,.
		\end{split}
	\end{equation*}
	In the same way, since $a_+ \leq |a|$ for $a \in \bbR$ we have
	\begin{equation*}
		\cL^{2n^2}(\cN^{\lambda}_{i}) \leq \frac{1}{\lambda^p}  \iintdm{B}{B}{ |x-y|^p |\rme[v](x,y)|^p }{y}{x} \qquad \text{ for all } 1 \leq i \leq 2n\,.
	\end{equation*}
	Therefore there exists $C' = C'(n)$ such that
	\begin{equation*}
		\cL^{2n^2} (A^{\lambda}) \leq \frac{C'}{\lambda^p} \iintdm{B}{B}{ |x-y|^p |\rme[v](x,y)|^p }{y}{x}\,.
	\end{equation*}
	Now, setting $E := \left( \iintdm{B}{B}{ |x-y|^p |\rme[v](x,y)|^p }{y}{x} \right)^{1/p}$, choose
	\begin{equation}\label{eq:ChosenLevelSet}
		\lambda = \left( \frac{2C'}{\cL^{2n^2}(A) }+1 \right) E\,.
	\end{equation}
	Then since $\frac{E}{\lambda} \leq 1$
	\begin{equation*}
		C'\left( \frac{E}{\lambda} \right)^p \leq C'\left( \frac{E}{\lambda} \right) \leq \frac{1}{2} \cL^{2n^2}(A)\,.
	\end{equation*}
	Thus $\cL^{2n^2}(A^{\lambda}) \leq \frac{1}{2} \cL^{2n^2}(A)$, and so there must exist $\vec{x} \in A \setminus A^{\lambda}$. Thus for this choice of $\vec{x}$ we have by definition of $A^{\lambda}$
	\begin{equation*}
		\big| |v(x_i)-v(x_j)| - |x_i-x_j| \big| \leq \lambda\,, \qquad 1 \leq i,j \leq 2n,
	\end{equation*}
	and
	\begin{equation*}
		\left( \intdm{B''}{ \big( |v(x_i)-v(y) | - |x_i-y| \big)_+^p}{y} \right)^{1/p} \leq \lambda\,, \qquad 1 \leq i \leq 2n\,.
	\end{equation*}
	Now we apply Lemma \ref{lma:GeometricLemma} with $y_i = v(x_i)$, $\veps_1 = \lambda$, and $w = v(z)$ to obtain a rigid motion $\psi \in \scR$ such that
	\begin{equation*}
		|v(x_i)-\psi(x_i)| \leq C_0 \lambda\,, \qquad 1 \leq i \leq 2n,
	\end{equation*}
	and for all $z \in B''$ 
	\begin{equation*}
		|v(z)-\psi(z)| \leq C_0 \left( \lambda + \sum_{i=1}^{2n} \big( |v(z)-v(x_i)| - |z-x_i| \big)_+ \right)\,.
	\end{equation*}
	Integrating the last inequality over $z \in B''$ and using that $\vec{x} \in A \setminus A^{\lambda}$:
	\begin{equation*}
		\begin{split}
			\left( \intdm{B''}{|v(z)-\psi(z)|^p}{z} \right)^{1/p} &\leq C(n,p) \left(  |B''| \lambda + \sum_{i=1}^{2n} \left( \intdm{B''}{ \big( |v(x_i)-v(z) | - |x_i-z| \big)_+^p}{z} \right)^{1/p}   \right) \\
			&\leq C(n,p) \lambda\,.
		\end{split}
	\end{equation*}
	Thus by the value of $\lambda$ chosen in \eqref{eq:ChosenLevelSet} we have proved that
	\begin{equation}\label{eq:Scaled:NoDiff:Pf1}
		\intdm{B(0,1/M)}{|v(x)-\psi(x)|^p}{x} \leq N \iintdm{B(0,1)}{B(0,1)}{ |x-y|^p |\rme[v](x,y)|^p }{y}{x}
	\end{equation}
	with positive constants $N = N(n,p)$ and $M = 10(1+r_0)$.
	
	To finish the proof for arbitrary $x_0 \in \bbR^n$ and $r>0$ we use a translation and rescaling argument.
	let $M = 10(1+r_0)$, and suppose $v \in L^p(B(x_0,Mr);\bbR^n)$. 
	Then $w(x) = (Mr)^{-1} v(x_0+Mrx)$ belongs to $L^p(B(0,1);\bbR^n)$, and there exists $\bar{\psi} \in \scR$ such that \eqref{eq:Scaled:NoDiff:Pf1} holds for $w$. Then a change of variables shows that \eqref{eq:Scaled:NoDiff} holds for $v$ with rigid motion $\psi(x) = Mr \bar{\psi}((x-x_0/(Mr))$.
\end{proof}

We adapt this estimate to more general sets, so that we can later apply a localization argument for $\Omega$. We first need an elementary lemma comparing the Euclidean norm to a quantity related to the strain.

\begin{lemma}\label{lma:KeyElementaryEstimate}
	Let $p \in [1,\infty)$, $z \in \bbR^n$, $r > 0$ and $\alpha \in (0,1)$. Then there exists $C = C(n,\alpha) > 0$ such that for all $w \in \bbR^n$ and all $x \in B(z,2r)$
	\begin{equation}\label{eq:KeyElementaryEstimate}
		|w-x|^p \leq C r^{-n} \int_{B(z,\alpha r)} \big| |w-y| - |x-y| \big|^p \, \rmd y.
	\end{equation}
\end{lemma}

\begin{proof}
	It suffices to prove \eqref{eq:KeyElementaryEstimate} for $z = 0$, $r = 1$ and $p = 1$; the estimate for general $z$ and $r$ will follow by rescaling, and for general $p$ from Jensen's inequality. First assume that $|w-x| \geq 12$. Then since $|x-y| \leq 3$,
	\begin{equation*}
		\begin{split}
			|w-x| \leq |w-y| + |x-y| &= |w-y| - |x-y| + 2|x-y| \\
			&\leq \big| |w-y| - |x-y| \big| + 6 \leq \big| |w-y| - |x-y| \big| + \frac{1}{2}|w-x|.
		\end{split}
	\end{equation*}
	Absorbing $|w-x|$ onto the left-hand side and integrating in $y$, we obtain \eqref{eq:KeyElementaryEstimate}.
	
	Now suppose that $|w-x| < 12$. 
	Then $|w-y| + |x-y| \leq 18$ for all $y \in B(0,\alpha)$, and so
	\begin{equation}\label{eq:ElementaryEst:Pf1}
		\begin{split}
			\int_{B(0,\alpha)} \big| |w-y| - |x-y| \big| \, \rmd y 
				&= \int_{B(0,\alpha)} \frac{\big| |w-y|^2 - |x-y|^2 \big|}{|w-y| + |x-y|} \, \rmd y \\
				&\geq \frac{1}{18} \int_{B(0,\alpha)} \big| |w-y|^2 - |x-y|^2 \big| \, \rmd y  \\
				&= \frac{1}{18} \int_{B(0,\alpha)} \big| 2(w-x) \cdot y + |w|^2 - |x|^2 \big| \, \rmd y.
		\end{split}
	\end{equation}
	We note that for any affine function of the form $a \cdot x + b$ where $a$, $b \in \bbR^n$
	\begin{equation*}
		|a| \leq C(n,\alpha) \int_{B(0,\alpha)} |a \cdot y + b| \, \rmd y.
	\end{equation*}
	(This can be seen by applying \eqref{eq:AffineMapEquivalence:Rescaled:MatrixPart} with $A = a \otimes \nu$ for some $\nu \in \bbS^{n-1}$). Applying this to the right-hand side of \eqref{eq:ElementaryEst:Pf1} with $a = 2(w-x)$ and $b = |w|^2 - |x|^2$, we obtain \eqref{eq:KeyElementaryEstimate}.
\end{proof}

\begin{theorem}\label{thm:GeoRigid:Lp:Ball}
	Let $\Omega \subset \bbR^n$ be a bounded Lipschitz domain and let $p \in [1,\infty)$. 
	Then there exists a constant $C = C(n,p,\Omega) > 0$ such that the following holds:
	For $x_0 \in \overline{\Omega}$ and $r \in (0,\diam\Omega)$ denote the set $\Omega_{x_0,r} := \Omega \cap B(x_0,r)$. For all $v \in L^p(\Omega;\bbR^n)$ there exists a rigid motion $\psi \in \scR$ such that
	\begin{equation}\label{eq:GeoRigid:Lp:Ball}
		\int_{\Omega_{x_0,r}} |v(x) - \psi(x)|^p \, \rmd x \leq C r^{-n} \int_{\Omega_{x_0,r}} \int_{\Omega_{x_0,r}}  |x-y|^p |\rme[v](x,y)|^p \, \rmd y \, \rmd x.
	\end{equation}
\end{theorem}

\begin{proof}
	We first note the following:
	\begin{equation}\label{eq:Lipschitz:ZoomIn}
		\begin{gathered}
			\text{There exists a constant } \kappa = \kappa(\Omega) > 0 \text{ such that, } \forall x_0 \in \overline{\Omega} \text{ and } r \in (0,2\diam\Omega), \\
			B_{x_0,r} \text{ contains a ball of radius } \kappa r. \text{ Moreover, } B(z,\kappa r) \subset B_{x_0,r} \subset B(z,2 r).
		\end{gathered}
	\end{equation}
	This is a consequence of the uniform interior cone condition satisfied by bounded Lipschitz domains; see \cite{MengeshaConsistency} for the details of the argument.
	
	We apply \cref{prop:Scaled:NoDiff} on the set $B(z,\kappa r)$ as given in \eqref{eq:Lipschitz:ZoomIn} to obtain a rigid motion $\psi$ such that
	\begin{equation}\label{eq:ScaledLipschitz:Pf1}
		\int_{B(z,\kappa r / M)} |v(x) - \psi(x)|^p \, \rmd x \leq C r^{-n} \int_{B(z,\kappa r)} \int_{B(z,\kappa r)} |x-y|^p |\rme[v](x,y)|^p \, \rmd y \, \rmd x.
	\end{equation}
	Define $\wt{v}(x) = A^{-1} v(x) - A^{-1}b$, where $\psi(x) = Ax + b$; by \eqref{eq:KeyElementaryEstimate} with $w = \wt{v}(x)$ and $\alpha = \kappa/M$  
	\begin{equation*}
		\begin{split}
			|\wt{v}(x) - x|^p &\leq C r^{-n} \int_{B(z,\kappa r/M)} \big| |\wt{v}(x) - y| - |x-y| \big|^p \, \rmd y \\
			&\leq C r^{-n} \int_{B(z,\kappa r/M)} \big| |\wt{v}(x) - \wt{v}(y)| - |x-y| \big|^p \, \rmd y + C r^{-n} \int_{B(z,\kappa r/M)} |\wt{v}(y) - y|^p \, \rmd y \\
			&= C r^{-n} \int_{B(z,\kappa r/M)} \big| |v(x) - v(y)| - |x-y| \big|^p \, \rmd y + C r^{-n} \int_{B(z,\kappa r/M)} |v(y) - \psi(y)|^p \, \rmd y \\
		\end{split}
	\end{equation*}
	where the equality follows from the fact that $A^T A = I$. The left-hand side also becomes $|v(x) - \psi(x)|^p$. The inequality is valid for all $x \in \Omega_{x_0,r}$; integrating over $x$ and using that $|\Omega_{x_0,r}| \leq C(n) r^n$
	\begin{equation*}
		\vnorm{v - \psi}_{L^p(\Omega_{x_0,r})}^p \leq C r^{-n} \int_{\Omega_{x_0,r}} \int_{\Omega_{x_0,r}} |x-y|^p |\rme[v](x,y)|^p \, \rmd y \, \rmd x + C \vnorm{v - \psi}_{L^p(B(z,\kappa r/M))}^p.
	\end{equation*}
	The result follows by using \eqref{eq:ScaledLipschitz:Pf1} to estimate the $L^p$ norm on the right-hand side.
\end{proof}

\begin{lemma}\label{lma:LpEst:Annulus}
	In the setting of \cref{thm:GeoRigid:Lp:Ball}, there exists $\lambda \in (0,1)$ depending only on $n$, $p$, and $\Omega$ such that
	\begin{equation}
		\int_{\Omega_{x_0,r}} |v(x) - \psi(x)|^p \, \rmd x \leq C r^{-n} \int_{\Omega_{x_0,r}} \int_{\Omega_{x_0,r}} \mathds{1}_{ \{ \lambda r < |x-y| \leq 2 r \} } |x-y|^p |\rme[v](x,y)|^p \, \rmd y \, \rmd x.
	\end{equation}
\end{lemma}

\begin{proof}
	Starting from \eqref{eq:GeoRigid:Lp:Ball},
	write $1 = \mathds{1}_{  \{ |x-y| \leq \lambda r \} } +  \mathds{1}_{  \{ \lambda r < |x-y| \leq 2 r \} }$ and split the right-hand side of \eqref{eq:GeoRigid:Lp:Ball} accordingly. The first term will be absorbed on the left-hand side. 
	To this end, since $\psi$ is a rigid motion, we have $|x-y| |\rme[v](x,y)| = ||v(x)-v(y)| - |\psi(x-y)| | \leq |v(x) - \psi(x)| + |v(y) - \psi(y)|$, and so
	\begin{equation*}
		\int_{\Omega_{x_0,r}} \int_{\Omega_{x_0,r}} \mathds{1}_{ \{ |x-y| \leq \lambda r \} } |x-y|^p |\rme[v](x,y)|^p \, \rmd y \, \rmd x \leq 2^p \cL^n(B(0,1)) (\lambda r)^n \int_{\Omega_{x_0,r}} |v(x)-\psi(x)|^p \, \rmd x.
	\end{equation*}
	For the constant $C$ in \eqref{eq:GeoRigid:Lp:Ball}, choose $\lambda > 0$ so that $2^p C \cL^n(B(0,1)) \lambda^n < 1/2$. The result follows.
\end{proof}

Define for a bounded Lipschitz domain $\Omega \subset \bbR^n$, $\lambda$ as in \cref{lma:LpEst:Annulus}, $p \in [1,\infty)$, $r \in (0,\diam(\Omega))$, $v : \Omega \to \bbR^n$ and any measurable set $G \subseteq \Omega$
$$
\begin{aligned}
	\cA_r(v;G)&:=\int_G \int_G
	\mathds{1}_{\{r/2<|x-y| \leq r\}}|v(x)-v(y)|^p\, \rmd y \,\rmd x,\\
	\cE_r(v;G)&:=\int_G \int_G
	\mathds{1}_{\{\lambda r <|x-y| \leq 2 r \}} \big| |v(x) - v(y)| - |x-y| \big|^p \, \rmd y \, \rmd x.
\end{aligned}
$$
We will use $\cE_r$ to estimate $\cA_r$ at different scales. However, because \cref{thm:Intro:GeometricRigidity} is a genuinely nonlocal result with no local analogue, we cannot expect an estimate of the type $\cA_r(v-\psi;\Omega) \leq C \cE_r(v;\Omega)$ to hold. Indeed, dividing both sides by $r^{n+p}$ and Taylor expanding $v$, one would obtain a local rigidity estimate for $\grad v$ near $\rmO(n)$, which is false. Therefore any bound of $\cA_r$ will need to involve all scales of the strain-type energy, not just $\cE_r$ at the single scale $r$. The following theorem makes this precise.

\begin{theorem}\label{thm:GeoRigid:Lp:Annulus}
	There exists a constant $C = C(n,p,\Omega) > 0$ such that the following holds: for every $v \in L^p(\Omega; \bbR^n)$ there exists $\psi \in \scR$ such that
	\begin{equation*}
		\cA_{r_k}(v-\psi;\Omega)^{1/p} \leq C \sum_{j=0}^k \left( \frac{2^{-k(n+p)}}{2^{-j(n+p)}} \cE_{r_j}(v;\Omega)  \right)^{1/p}
	\end{equation*}
	for all $k \in \bbN_0$, where $r_k = 2^{-k} \diam(\Omega)$.
\end{theorem}

\begin{proof}
	We choose a sequence of open covers $\cG_k$ of $\Omega$ as follows:
	For $k = 0$ set $\cG_0 = \{ \Omega \}$, and for $k \geq 1$
	choose a finite subcover $\{ B(x_{\ell},r_k/8) \}_{\ell=1}^K$ of $\overline{\Omega}$, where each $x_{\ell} \in \Omega$, and such that $B(x_\ell,r_k/16) \cap B(x_m,r_k/16) = \emptyset$ for any $\ell$, $m \in \{ 1, \ldots K \}$. Define $\cG_k := \{ B(z,r_k/2) \, : \, z \in \{ x_1, \ldots, x_K \} \}$. The dilations are denoted $R U = B(z,R r_k/2)$ for $U \in \cG_k$. The following properties of the $\cG_k$ are readily verified:
	\begin{itemize}
		\item[\textbf{i)}] For all $k \in \bbN_0$, the dilated collection $\{ 3 U \, : \, U \in \cG_K \}$ is locally finite, and there exists $C = C(n)$ such that each point of $\Omega$ is contained in at most $C$ of the elements of $\cG_k$.
		
		\item[\textbf{ii)}] $\{ (x,y) \in \Omega \times \Omega \, : \, r_{k+1} < |x-y| < r_k \} \subset \bigcup_{U \in \cG_k}  3 U \times 3 U$.
		
		\item[\textbf{iii)}] If $U \in \cG_{k+1}$ and $\wt{U} \in \cG_k$ satisfy $U \cap \wt{U} \neq \emptyset$, then $B(z,r_k/4) \subset 3 U \cap 3 \wt{U}$, where $z$ is the center of $U$.
		
		\item[\textbf{iv)}] There exists a constant $C = C(n) > 0$ such that for all $k \in \bbN_0$ the following holds: for any $U \in \cG_{k+1}$ and any $\wt{U} \in \cG_{k}$, the collections $\{ \wt{V} \in \cG_k \, : \, \wt{V} \cap U \neq \emptyset \}$ and $\{ V \in \cG_{k+1} \, : \, V \cap \wt{U} \neq \emptyset \}$ have cardinality at most $C$.
	\end{itemize}
	These covers will be used to define successive approximations up to scale $r_k$. 
	First, let $\psi_0 \in \scR$ be as in \cref{lma:LpEst:Annulus} with $r = \diam(\Omega)$ and some $x_0 \in \Omega$, so that
	\begin{equation}
		\vnorm{v - \psi_0}_{L^p(\Omega)}^p \leq C(n,p,\Omega) \int_{\Omega} \int_\Omega \mathds{1}_{ \{ \lambda \diam(\Omega) < |x-y| \} } |x-y| |\rme[v](x,y)|^p \, \rmd y \, \rmd x.
	\end{equation}
	Next, for each $j \in \{ 1, \ldots, k \}$, and each $U \in \cG_j$, apply  \cref{lma:LpEst:Annulus} to each set $\Omega \cap 3 U$ to obtain $\psi_U = A_U x + b_U \in \scR$ satisfying
	\begin{equation}\label{eq:GeoRigid:Lp:Annulus:Pf1}
		\vnorm{v-\psi_U}_{ L^p(\Omega \cap 3 U) }^p \leq C(n,p,\Omega) r_j^{-n} \cE_{r_j}(v;\Omega \cap 3 U). 
	\end{equation}
	Now for each $j \geq 1$, consider a $C^\infty$ partition of unity subordinate to the cover $\cG_j$, denoted as $\{ \varphi_U \}_{U \in \cG_j}$, with $|\grad \varphi_U| \leq C(n) r_j^{-1}$ for all $U \in \cG_j$ and all $j$. Define the functions $v_j$, $j = 0,\ldots k$, as $v_0 = \psi_0$ and $v_j(x) = \sum_{ U \in \cG_j } \varphi_U(x) \psi_U(x)$ for $j \geq 1$. 
	Then each $v_j \in C^\infty(\overline{\Omega})$, and by the local finiteness of the $\cG_j$,  \eqref{eq:GeoRigid:Lp:Annulus:Pf1} and property ii) we have
	\begin{equation}\label{eq:GeoRigid:Lp:Annulus:Pf2}
		\vnorm{v - v_j}_{L^p(\Omega)} \leq C(n,p,\Omega) r_j^{-n} \cE_{r_j}(v;\Omega).
	\end{equation}
	
	Now we begin to estimate $\cA_{r_k}$. Writing
	\begin{equation*}
		\begin{split}
			v(x) - v(y) - (\psi_0(x) -\psi_0(y)) &= v(x) - v(y) -  (v_k(x) - v_k(y)) \\
			&\qquad + \sum_{j=0}^{k-1} (v_{j+1}(x-y) - v_j(x-y)),
		\end{split}
	\end{equation*}
	we have
	\begin{equation}\label{eq:GeoRigid:Lp:Annulus:Pf3}
		\begin{split}
			\cA_{r_k}(v-\psi_0;\Omega)^{1/p} 
			&\leq \cA_{r_k}(v-v_k;\Omega)^{1/p} + \sum_{j=0}^{k-1} \cA_{r_k}(v_{j+1}-v_j;\Omega)^{1/p}.
		\end{split}
	\end{equation}
	First, by \eqref{eq:GeoRigid:Lp:Annulus:Pf2} 
	\begin{equation*}
		\cA_{r_k}(v-v_k;\Omega) \leq 2^p \int_{\Omega} |v(x) - v_k(x)|^p \, \rmd x \int_{ \{ |z| < r_k \} }\, \rmd z \leq C \cE_{r_k}(v;\Omega).
	\end{equation*}
	Second, by \cite[Theorem 1]{Bourgain2001Another}
	\begin{equation*}
		\begin{split}
			\cA_{r_k}(v_{j+1} - v_j;\Omega) &\leq C(p,\Omega) \int_{  \{ |z| < r_k  \} } |z|^p \, \rmd z \int_\Omega |\grad(v_{j+1} - v_j)|^p \, \rmd x \\
			&\leq C r_k^{n+p}  \vnorm{\grad(v_{j+1} - v_j)}_{L^p(\Omega)}^p.
		\end{split}
	\end{equation*}
	Substitute these two bounds into \eqref{eq:GeoRigid:Lp:Annulus:Pf3}; if we additionally show that
	\begin{equation}\label{eq:GeoRigid:Lp:Annulus:Pf4}
		\vnorm{\grad(v_{j+1} - v_j)}_{L^p(\Omega)}^p \leq C(n,p,\Omega) r_j^{-n-p} \left( \cE_{r_{j+1}}(v;\Omega) + \cE_{r_{j}}(v;\Omega) \right), \quad j \in \{ 0, \ldots, k -1 \},
	\end{equation}
	then the theorem will be proved with $\psi = \psi_0$ after applying the inequality $(a+b)^{1/p} \leq a^{1/p} + b^{1/p}$ for $a$, $b \geq 0$.
	
	To this end, using that $\sum_{U \in \cG_{j+1}} \varphi_U = \sum_{\wt{U} \in \cG_{j}} \varphi_{\wt{U}} = 1$ we write\\ $v_{j+1} - v_j = \sum_{U \in \cG_{j+1}} \sum_{\wt{U} \in \cG_j} \varphi_U \varphi_{\wt{U}} ( \psi_U - \psi_{\wt{U}})$, and so
	\begin{equation*}
		\grad (v_{j+1} - v_j) = \sum_{U \in \cG_{j+1}} \sum_{\wt{U} \in \cG_j} ( \psi_U - \psi_{\wt{U}}) \otimes ( \varphi_U \grad \varphi_{\wt{U}} + \varphi_{\wt{U}} \grad \varphi_U ) + \varphi_U \varphi_{\wt{U}} ( A_U - A_{\wt{U}}).
	\end{equation*}
	Then using property iv) of the covers and the bounds $|\varphi_U| \leq 1$ and $|\grad \varphi_U| \leq C r_j^{-1}$,
	\begin{equation*}
		\vnorm{\grad(v_{j+1} - v_j)}_{L^p(\Omega)}^p \leq C(n,p) \sum_{ \wt{U} \in \cG_j } \sum_{\substack{U \in \cG_{j+1} \\ U \cap \wt{U} \neq \emptyset }} \int_{\Omega \cap \wt{U}} \frac{|\psi_U - \psi_{\wt{U}}|^p}{r_j^p} + |A_U - A_{\wt{U}}|^p \, \rmd x.
	\end{equation*}
	Let $z$ be the center of $U$; property iii) of $\cG_j$ implies that $B(z,r_j/4) \subset 3 U \cap 3 \wt{U}$. Then both $B(z,r_j/4)$ and $\wt{U}$ are contained in a ball of radius comparable to $r_j$, and the centers of all three are contained in $\Omega$. By an argument similar to that of \cref{lma:PolynomialsOnSets}, we can replace $\vnorm{\psi_U - \psi_{\wt{U}}}_{L^p(\Omega \cap \wt{U})}$ with $\vnorm{\psi_U - \psi_{\wt{U}}}_{L^p(\Omega \cap B(z,r_j/4) )}$ at the expense of a constant depending only on $p$, $n$, and $\Omega$.
	Additionally, by an argument similar to  \eqref{eq:AffineMapEquivalence:Rescaled:MatrixPart} we have $|A_U - A_{\wt{U}}|^p \leq C r^{-n-p} \vnorm{\psi_U - \psi_{\wt{U}}}_{L^p(\Omega \cap B(z,r_j/4) )}^p$. Therefore by \eqref{eq:GeoRigid:Lp:Annulus:Pf1}
	\begin{equation*}
		\begin{split}
			\vnorm{\grad(v_{j+1} - v_j)}_{L^p(\Omega)}^p &\leq C r_j^{-p} \sum_{ \wt{U} \in \cG_j } \sum_{\substack{U \in \cG_{j+1} \\ U \cap \wt{U} \neq \emptyset }}  \vnorm{\psi_U - \psi_{\wt{U}}}_{L^p(\Omega \cap B(z,r_j/4) )} \\
			&\leq C r_j^{-p} \sum_{ \wt{U} \in \cG_j } \sum_{\substack{U \in \cG_{j+1} \\ U \cap \wt{U} \neq \emptyset }}  \vnorm{v - \psi_U}_{L^p(\Omega \cap 3 U)} + \vnorm{v- \psi_{\wt{U}}}_{L^p(\Omega \cap 3 \wt{U}) } \\
			&\leq C r_j^{-n-p} \sum_{ \wt{U} \in \cG_j } \sum_{\substack{U \in \cG_{j+1} \\ U \cap \wt{U} \neq \emptyset }} \cE_{r_{j+1}}(v;\Omega \cap 3 U) + \cE_{r_{j}}(v;\Omega \cap 3 \wt{U}).
		\end{split}
	\end{equation*}
	To estimate the sums, we use property iv), Tonelli's theorem, and property i) to arrive at \eqref{eq:GeoRigid:Lp:Annulus:Pf4} as desired.
\end{proof}

\subsection{Rigidity for vector fields in the nonlocal space}

We begin with a Poincar\'e-type inequality.

\begin{theorem}\label{thm:GeometricRigidity:Poincare}
	With the assumptions of \cref{thm:Intro:GeometricRigidity}, there exist constants $C = C(n,p,\rho,\Omega) > 0$ and $\Lambda = \Lambda(n,p,\rho,\Omega) \in (0,r_0)$ such that the following holds: for all $v \in \cS^{\rho,p}(\Omega;\bbR^n)$ there exists $\psi \in \scR$ such that
	\begin{equation}
		\vnorm{v - \psi}_{L^p(\Omega)}^p \leq C \int_{\Omega} \int_{\Omega} \mathds{1}_{ \{ \Lambda < |x-y| \} }\rho(x-y) |x-y|^p |\rme[v](x,y)|^p \, \rmd y \, \rmd x.
	\end{equation}
\end{theorem}

\begin{proof}
	Let $\psi_0 \in \scR$ be as in \cref{lma:LpEst:Annulus} with $r = \diam(\Omega)$ and some $x_0 \in \Omega$, so that
	\begin{equation}\label{eq:Poincare:Pf1}
		\vnorm{v - \psi_0}_{L^p(\Omega)}^p \leq C(n,p,\Omega) \int_{\Omega} \int_\Omega \mathds{1}_{ \{ \lambda \diam(\Omega) < |x-y| \} } |x-y| |\rme[v](x,y)|^p \, \rmd y \, \rmd x.
	\end{equation}
	Recall that a nonzero nonincreasing function is locally of bounded variation with at most countably many discontinuities. 
	Hence there exists $R < r_0$, where $r_0$ is as in \eqref{assu:GeomRigid}, such that $\rho(r) > 0$ for $r \leq R$.
	Cover $\overline{\Omega}$ with finitely sets of the form $\Omega_{x_i,R}$, $i = 1,\ldots,N$, $N$ depending only on $\Omega$ and $r$, and enumerated such that $\Omega_{x_i,R} \cap \Omega_{x_{i+1},R} \neq \emptyset$.
	Apply \cref{lma:LpEst:Annulus} to each set $\Omega_{x_i,R}$ to obtain a rigid motion $\psi_i$ such that
	\begin{equation}\label{eq:Poincare:Pf2}
		\begin{split}
			&\vnorm{v - \psi_i}_{L^p(\Omega_{x_i,R})}^p \\
			\leq & C(n,p,\Omega,R) \int_{\Omega_{x_i,R}} \int_{\Omega_{x_i,R}} \mathds{1}_{ \{ \lambda R < |x-y| < 2R \} } |x-y|^p |\rme[v](x,y)|^p \, \rmd y \, \rmd x \\
			\leq & \frac{C}{\rho(\lambda R)} \int_{\Omega_{x_i,R}} \int_{\Omega_{x_i,R}} \mathds{1}_{ \{ \lambda R < |x-y| < 2R \} } \rho(|x-y|) |x-y|^p |\rme[v](x,y)|^p \, \rmd y \, \rmd x,
		\end{split}
	\end{equation}
	where we also used that $\rho$ is nonincreasing.
	For brevity we write $\Omega_i = \Omega_{x_i,R}$ for $i = 1,\ldots, N$, and begin to estimate
	\begin{equation}\label{eq:Poincare:Pf3}
		\begin{split}
			\vnorm{v - \psi_0}_{L^p(\Omega)}^p \leq C(N) \sum_{i=1}^N \vnorm{v - \psi_0}_{L^p(\Omega_{i})}^p &\leq C \sum_{i=1}^N \vnorm{v - \psi_i}_{L^p(\Omega_{i})}^p + \vnorm{\psi_i-\psi_0}_{L^p(\Omega_{i})}^p \\
			&\leq C(R) \sum_{i=1}^N \vnorm{v - \psi_i}_{L^p(\Omega_{i})}^p + \vnorm{\psi_i-\psi_0}_{\mathrm{aff}}^p,
		\end{split}
	\end{equation}
	where in the last inequality we used an estimate similar to \eqref{eq:AffineMapEquivalence:Rescaled}.
	By \cref{lma:PolynomialsOnSets} applied on a ball $B_k \subset \Omega_{k} \cap \Omega_{k+1}$ and \eqref{eq:Poincare:Pf2}
	\begin{equation*}
		\begin{split}
			\vnorm{\psi_{k+1}-\psi_k}_{\mathrm{aff}} &\leq C(n,p,\Omega) \vnorm{\psi_{k+1}-\psi_k}_{L^p(B_k)} \\
			&\leq C \vnorm{ v - \psi_{k+1}}_{L^p(\Omega_{k+1})} + C \vnorm{ v - \psi_{k}}_{L^p(\Omega_k)} \\
			&\leq C(n,p,\rho,\Omega) \int_{\Omega} \int_{\Omega} \mathds{1}_{ \{ \lambda R < |x-y| < 2R \} } \rho(|x-y|) |x-y|^p |\rme[v](x,y)|^p \, \rmd y \, \rmd x,
		\end{split}
	\end{equation*}
	for each $k \in \{1, \ldots N - 1 \}$. Then by \eqref{eq:Poincare:Pf1}-\eqref{eq:Poincare:Pf2} we have for $\Lambda = \lambda \min \{ R, \diam(\Omega) \}$
	\begin{equation*}
		\begin{split}
			\vnorm{\psi_i-\psi_0}_{\mathrm{aff}}
			&\leq C \vnorm{v - \psi_0}_{L^p(\Omega)} + C \vnorm{v - \psi_1}_{L^p(\Omega_{1})}  + \sum_{k=1}^{i-1} \vnorm{\psi_{k+1}-\psi_k}_{\mathrm{aff}} \\
			&\leq C(n,p,\rho,\Omega) \int_{\Omega} \int_{\Omega} \mathds{1}_{ \{ \Lambda < |x-y| \} } \rho(|x-y|) |x-y|^p |\rme[v](x,y)|^p \, \rmd y \, \rmd x.
		\end{split}
	\end{equation*}
	Therefore using this estimate in \eqref{eq:Poincare:Pf3} and again using \eqref{eq:Poincare:Pf2} we obtain the result.
	
\end{proof}

\begin{proof}[Proof of \cref{thm:Intro:GeometricRigidity}]
	Observe that the same $\psi \in \scR$ can be chosen in both conclusions of \cref{thm:GeoRigid:Lp:Annulus} and \cref{thm:GeometricRigidity:Poincare}; we choose such a $\psi$ here.
	We first assume that $r_0 \geq \diam(\Omega)$.
	Using that $\rho$ is nonincreasing,
	\begin{equation}\label{eq:GeometricRigidity:PfBeginning}
		\begin{split}
			\int_\Omega \int_\Omega \rho(x-y) |v(x)-v(y) - (\psi(x)-\psi(y))|^p \, \rmd y \, \rmd x \leq \sum_{k=0}^\infty \rho(r_{k+1}) \cA_{r_k}(v-\psi;\Omega).
		\end{split}
	\end{equation}
	By \cref{thm:GeoRigid:Lp:Annulus}
	\begin{equation*}
		[v-\psi]_{\cS^{\rho,p}(\Omega)}^p \leq C \sum_{k=0}^\infty \rho(r_{k+1}) \left( \sum_{j=0}^k \left( \frac{2^{-k(n+p)}}{2^{-j(n+p)}} \cE_{r_j}(v;\Omega)  \right)^{1/p} \right)^p
	\end{equation*}
	Let $\beta \in \bbR$ be such that $q < 2^\beta < 2^{n+p}$. Then H\"older's inequality gives
	\begin{equation*}
		\sum_{j=0}^k \left( 2^{j(n+p)} \cE_{r_j}(v;\Omega)  \right)^{1/p} \leq \left( \sum_{j=0}^k 2^{\frac{j(n+p-\beta)}{p-1}} \right)^{\frac{p-1}{p}} \sum_{j=0}^k 2^{j \beta} \cE_{r_j}(v;\Omega).
	\end{equation*}
	The first sum can be majorized by $C(n,p) 2^{ \frac{k(n+p-\beta)}{p} }$, so therefore using Tonelli's theorem
	\begin{equation*}
		[v-\psi]_{\cS^{\rho,p}(\Omega)}^p \leq C \sum_{k=0}^\infty \rho(r_{k+1})  \sum_{j=0}^k 2^{-\beta(k-j)} \cE_{r_j}(v;\Omega)
		\leq C \sum_{j=0}^\infty \sum_{k=j}^\infty \rho(r_{k+1}) 2^{-\beta(k-j)} \cE_{r_j}(v;\Omega).
	\end{equation*}
	Now we use \eqref{assu:GeomRigid}; for each $k \geq j$, we apply the doubling condition $k - j$ times: $\rho(r_{k+1}) \leq q \rho(r_k) \leq q^2 \rho(r_{k-1}) \leq \ldots \leq q^{k-j} \rho(r_{j+1})$. Therefore
	\begin{equation*}
		[v-\psi]_{\cS^{\rho,p}(\Omega)}^p \leq C \sum_{j=0}^\infty \rho(r_{j+1}) \cE_{r_j}(v;\Omega) \sum_{k=j}^\infty  (q 2^{-\beta} )^{k-j}.
	\end{equation*}
	The series in $k$ is convergent since $q < 2^\beta$, and therefore to complete the theorem we need to show that
	\begin{equation}\label{eq:Intro:GeometricRigidity:Pf1}
		\sum_{j=0}^\infty \rho(r_{j+1}) \cE_{r_j}(v;\Omega) \leq C(n,p,\rho,\Omega) \int_{\Omega} \int_{\Omega} \rho(|x-y|) \big| |v(x)-v(y)| - |x-y| \big|^p \, \rmd  y \, \rmd x.
	\end{equation}
	On the domain of integration of $\cE_{r_j}$, we have that $|x-y| \leq r_{j-1}$. Thus since $\rho$ is nonincreasing $\rho(r_{j+1}) \leq q^2 \rho(r_{j-1}) \leq q^2 \rho(x-y)$, and
	\begin{equation*}
		\rho(r_{j+1}) \cE_{r_j}(v;\Omega) \leq q^2 \int_\Omega \int_\Omega \mathds{1}_{ \{ \lambda r_j < |x-y| < r_{j-1} \} } \rho(x-y) \big| |v(x)-v(y)| - |x-y| \big|^p \, \rmd  y \, \rmd x.
	\end{equation*}
	Even though the annuli are enlarged in the definition of $\cE_{r_j}(v;\Omega)$, the sets $ \{ (x,y) \in \Omega \times \Omega \, : \, \lambda r_j \leq |x-y| \leq r_{j-1} \} $ still form a locally finite cover of $\Omega \times \Omega$, with any $(x,y)$ contained in at most $N$ sets with $N$ dependent only on $n$, $p$, and $\Omega$. 
	Therefore \eqref{eq:Intro:GeometricRigidity:Pf1} holds.
	
	Now we assume that $r_0 < \diam(\Omega)$. Let $\Lambda \in (0,r_0)$ be as in \cref{thm:GeometricRigidity:Poincare}. Choose $K \in \bbN$ such that $r_{K+1} < \Lambda/8 < r_K$. 
	Then write $1 = \mathds{1}_{ \{ |x-y| > \Lambda/8 \} } + \mathds{1}_{ \{ |x-y| < \Lambda/8 \} } $, and split the integral defining $[v-\psi]_{\cS^{\rho,p}(\Omega)}^p$ accordingly. The first term can be estimated by
	\begin{equation*}
		\frac{2^p}{ (\Lambda/8)^p } \vnorm{v-\psi}_{L^p(\Omega)}^p \int_{B(0,\diam(\Omega))} |z|^p \rho(z) \, \rmd z \leq C \int_{\Omega} \int_{\Omega} \mathds{1}_{ \{ \Lambda < |x-y| \} }\rho(x-y) |x-y|^p |\rme[v](x,y)|^p \, \rmd y \, \rmd x
	\end{equation*}
	thanks to \cref{thm:GeometricRigidity:Poincare}. The second term can be estimated just as in \eqref{eq:GeometricRigidity:PfBeginning} with the sum beginning at $k=K$. Then an argument identical to the first part of the proof shows that the second term is majorized by
	\begin{equation*}
		C(n,p,\rho,\Omega) \int_{\Omega} \int_{\Omega} \mathds{1}_{ \{ |x-y| < \Lambda/2 \} }\rho(x-y) |x-y|^p |\rme[v](x,y)|^p \, \rmd y \, \rmd x.
	\end{equation*}
	Thus the theorem is proved in the general case.
\end{proof}

\noindent Noting that the same rigid motion $\psi$ appears in the conclusions of both \cref{thm:Intro:GeometricRigidity} and \cref{thm:GeometricRigidity:Poincare}, we have the following:
\begin{corollary}\label{cor:FullNormBd}
	Under the assumptions of \cref{thm:Intro:GeometricRigidity}, there exists $C = C(n,p,\rho,\Omega) > 0$ such that for every $v \in \cS^{\rho,p}(\Omega;\bbR^n)$ there exists $\psi \in \scR$ such that
	\begin{equation*}
		\vnorm{v - \psi}_{\cS^{\rho,p}(\Omega)}^p \leq C \int_\Omega \int_\Omega \rho(x-y) |x-y|^p |\rme[v](x,y)|^p \, \rmd y \, \rmd x.
	\end{equation*}
\end{corollary}

\section{Justification of Linearization}\label{sec:Justification}

In this section we prove \cref{thm:VariationalConvergence}. The proofs of \cref{thm:EnergyBound} and \cref{thm:GammaConvergence} are completed in the next respective sections.

\subsection{Compactness}

\begin{proof}[Proof of \cref{thm:EnergyBound}]
	Denote $v_\veps(x) = x + \veps u_\veps(x)$. 
	First, by \eqref{eq:AssumptionsOnW:3}
	\begin{equation*}
		\int_\Omega \int_\Omega \rho(x-y) |x-y|^2 |\rme[v_\veps](x,y)|^2 \, \rmd y \, \rmd x \leq \frac{1}{c_0} \veps^2 \cE_\veps(u_\veps).
	\end{equation*}
	\cref{cor:FullNormBd} implies that there exists a rigid motion $\psi_\veps$ depending on $v_\veps$ such that 
	\begin{equation}\label{eq:CompactnessPf:1}
		\vnorm{v_\veps - \psi_\veps}_{\cS^{\rho,2}(\Omega)}^2 \leq C \veps^2 \cE_\veps(u_\veps).
	\end{equation}
	Since $v_\veps(x) = x + \veps g(x)$ on $\omega$ for all $\veps$, and since $\omega \subset \Omega$, we have
	\begin{equation*}
		\vnorm{\mathrm{id} + \veps g - \psi_\veps}_{\cS^{\rho,2}(\omega)}^2 \leq C \veps^2 \cE_\veps(u_\veps),
	\end{equation*}
	implying that
	\begin{equation}\label{eq:CompactnessPf:2}
		\vnorm{\mathrm{id} - \psi_\veps}_{\cS^{\rho,2}(\omega)}^2  \leq C \veps^2 \left( \cE_\veps(u_\veps) + \vnorm{ g }_{\cS^{\rho,2}(\omega)}^2  \right).
	\end{equation}
	
	Now, $x - \psi_\veps(x)$ is an affine vector field, and so a straightforward rescaling argument similar to that of the proof of \eqref{eq:AffineMapEquivalence:Rescaled}
	shows that we can exchange $\vnorm{\mathrm{id} - \psi_\veps}_{\cS^{\rho,2}(\omega)}^2 $ with $\vnorm{\mathrm{id} - \psi_\veps}_{\cS^{\rho,2}(\Omega)}^2 $ at the cost of a constant depending on $n$, $\rho$, $\cL^n(\Omega)$, and $\cL^n(\omega)$. With this, along with \eqref{eq:CompactnessPf:1} and \eqref{eq:CompactnessPf:2},
	\begin{equation*}
		\begin{split}
			\veps^2 \vnorm{u_\veps}_{\cS^{\rho,2}(\Omega)}^2
			&\leq 2 \vnorm{\mathrm{id} + \veps u_\veps - \psi_\veps}_{\cS^{\rho,2}(\Omega)}^2 + 2 \vnorm{\mathrm{id} - \psi_\veps}_{\cS^{\rho,2}(\Omega)}^2 \\
			&\leq 2 \vnorm{v_\veps - \psi_\veps}_{\cS^{\rho,2}(\Omega)}^2 + C \vnorm{\mathrm{id} - \psi_\veps}_{\cS^{\rho,2}(\omega)}^2 \\
			&\leq C \veps^2 \left( \cE_\veps(u_\veps) + \vnorm{g}_{\cS^{\rho,2}(\omega)}^2 \right).
		\end{split}
	\end{equation*}
\end{proof}

\begin{theorem}\label{thm:Equicoercivity}
	With the assumptions of \cref{thm:EnergyBound}, assume additionally that $\rho$ satisfies \eqref{assu:Compact}.
	Then functionals $\overline{\cF}_\veps$ are equicoercive in $L^2(\Omega;\bbR^n)$.
\end{theorem}

\begin{proof}
	Suppose that a sequence $\{ u_\veps \}$ satisfies $\overline{\cF}_\veps(u_\veps) \leq C < \infty$; thus $\{ u_\veps \} \subset \cS^{\rho,2}_g$. By H\"older's inequality
	\begin{equation*}
		\cE_\veps(u_\veps) \leq C + \int_\Omega \ell \cdot u_\veps \, \rmd x \leq C (1+ \vnorm{u_\veps}_{L^2(\Omega)} ). 
	\end{equation*}
	Then this along with \cref{thm:EnergyBound} implies
	\begin{equation*}
		\vnorm{u_\veps}_{\cS^{\rho,2}(\Omega)}^2 \leq C ( \cE_\veps(u_\veps) + \vnorm{g}_{\cS^{\rho,2}(\omega)}^2) \leq C ( 1 + \vnorm{u_\veps}_{L^2(\Omega)} + C \vnorm{g}_{\cS^{\rho,2}(\omega)}^2) \leq C ( 1 + \vnorm{u_\veps}_{\cS^{\rho,2}(\Omega)}),
	\end{equation*}
	for a constant independent of $u_\veps$ and $\veps$. Thus $ \{ \vnorm{u_\veps}_{\cS^{\rho,2}(\Omega)} \}_\veps$ is bounded.
	Therefore, since $\rho$ satisfies \eqref{assu:Compact}, \cite[Theorem 1.3]{Du2023$} can be applied to conclude that $\{ u_\veps \}$ is precompact in $L^2(\Omega;\bbR^n)$. Moreover, any limit point belongs to $\cS^{\rho,2}(\Omega;\bbR^n)$, and thus the $\overline{\cF}_\veps$ are equicoercive.
\end{proof}

\begin{remark}
	Note that the required assumptions for \cite[Theorem 1.3]{Du2023$} follow from \eqref{assu:Kernel} and \eqref{assu:Compact}; since any nonzero nonincreasing function is locally of bounded variation with at most countably many discontinuities, there exists $R > 0$ such that $\rho(r) > 0$ for $r \in (0,R)$.
\end{remark}

\subsection{\texorpdfstring{$\Gamma$}{Gamma}-convergence}

The proof is similar to that of \cite[Theorem 1.2]{Mengesha2024Linearization}. However, that $\Gamma$-limit was computed without additional constraints and with stricter continuity assumptions on $W$. We provide arguments for the more general result here.

\begin{proof}[Proof of Theorem \ref{thm:GammaConvergence}]
	The result follows from the strong $L^2$-continuity of the functional $\scL$ and the inequalities
	\begin{equation*}
		\overline{\cE}(u) \leq \Gammaliminf_{\veps \to 0^+} \overline{\cE}_\veps(u) \leq \Gammalimsup_{\veps \to 0^+} \overline{\cE}_\veps(u) \leq \overline{\cE}(u)
	\end{equation*}
	for all $u \in L^2(\Omega;\bbR^n)$. The second inequality is trivial, and the first and third inequalities are proved in \cref{lma:Linearization:Liminf} and \cref{lma:Linearization:Limsup}, respectively.
\end{proof}

\begin{lemma}\label{lma:Linearization:Liminf}
	Let $u \in L^{2}(\Omega;\bbR^n)$.   For every sequence $u_{\veps}$ converging to $u$ strongly in $L^2(\Omega;\bbR^n)$, we have
	\begin{equation*}
		\overline{\cE}(u) \leq \liminf_{\veps \to 0^+} \overline{\cE}_\veps(u_\veps).
	\end{equation*}
\end{lemma}
\begin{proof}
	We may assume that $\liminf_{\veps \to 0^+} \overline{\cE}_\veps(u_\veps) < \infty$, otherwise the result is trivial. Then (without relabeling) there is a subsequence $\{ \overline{\cE}_\veps(u_\veps) \}$ that is bounded, hence $\overline{\cE}_\veps(u_\veps) = \cE_\veps(u_\veps)$ for each $\veps$ and so $\{ u_\veps \} \subset \cS^{\rho,2}_g$ by definition of $\overline{\cF}_\veps$. Moreover, by \cref{thm:EnergyBound} we have
	\begin{equation}\label{eq:GammaLimInf:Pf2}
		\vnorm{u_\veps}_{\cS^{\rho,2}(\Omega)}^2 \leq C, \qquad \forall \veps > 0.
	\end{equation}
	We need to show that
	\begin{equation}\label{eq:GammaLimInf:Pf1}
		\liminf_{\veps\to 0^+} \cE_\veps(u_\veps) \geq \cE(u)
	\end{equation}
	for any sequence $u_\veps$ converging to $u$ in $L^2(\Omega;\bbR^n)$ and almost everywhere with $\{ \cE_\veps(u_\veps) \}$ bounded.
	Since by Chebyshev's theorem
	\begin{equation*}
		\cL^{2n}( \{ (x,y) \in \Omega \times \Omega \, : \, |D[u_\veps](x,y)| > \veps^{-1/2} \} ) \to 0 \text{ as } \veps \to 0,
	\end{equation*}
	it follows that
	\begin{equation}\label{eq:GammaLimInf:Pf3}
		\mathds{1}_{\cB_\veps}(x,y) D[u_{\veps}](x,y)
		\to D[u](x,y) \quad \cL^{2n}\text{-a.e. in } \Omega \times \Omega,
	\end{equation}
	where
	\begin{equation*}
		\cB_\veps := \{ (x,y) \in \Omega \times \Omega \, : \, |D[u_\veps](x,y)| \leq  \veps^{-1/2} \}.
	\end{equation*}
	Now, from \eqref{eq:AssumptionsOnW:1} and \eqref{eq:AssumptionsOnW:2} it can be shown that for all $z \in \bbS^{n-1}$ and all $\xi \in \bbR^n$
	\begin{equation*}
		W \left( \frac{|z + \xi|}{|z|} - 1 \right) \geq \frac{\mu}{2} \left( \frac{\xi \cdot z}{|z|^2} \right)^2 - \alpha(|\xi|) |\xi|^2,
	\end{equation*}
	where $\alpha : [0,\infty) \to [0,\infty)$ is an increasing (not necessarily continuous) function such that $\alpha(t) \to 0$ as $t \to 0^+$. Therefore,
	\begin{equation*}
		\begin{split}
			\cE_\veps(u_\veps) &\geq \iint_{\cB_\veps} \frac{\mu}{2} |x-y|^2 \rho(x-y) ( D[u_\veps](x,y) \cdot D[\mathrm{id}](x,y))^2 - \alpha(|D[u_\veps](x,y)|)|D[u_\veps](x,y)|^2 \, \rmd y \, \rmd x \\
			&\geq \int_\Omega \int_\Omega  \frac{\mu}{2} |x-y|^2 \rho(x-y) \mathds{1}_{\cB_\veps} ( D[u_\veps](x,y) \cdot D[\mathrm{id}](x,y))^2 - \alpha(\sqrt{\veps}) \mathds{1}_{\cB_\veps} |D[u_\veps](x,y)|^2 \, \rmd y \, \rmd x.
		\end{split}
	\end{equation*}
	Since the last term is bounded from below by $-\alpha(\sqrt{\veps}) \vnorm{u_\veps}_{\cS^{\rho,2}(\Omega)}$, by \eqref{eq:GammaLimInf:Pf2} and since $\alpha(\sqrt{\veps}) \to 0$ as $\veps \to 0$,
	\begin{equation*}
		\begin{split}
			\liminf_{\veps \to 0^+} \cE_\veps(u_\veps) \geq \frac{\mu}{2} 	\liminf_{\veps \to 0^+} \int_\Omega \int_\Omega |x-y|^2 \rho(x-y) \mathds{1}_{\cB_\veps}(x,y) ( D[u_\veps](x,y) \cdot D[\mathrm{id}](x,y))^2 \, \rmd y \, \rmd x.
		\end{split}
	\end{equation*}
	Then by \eqref{eq:GammaLimInf:Pf3} and Fatou's lemma, \eqref{eq:GammaLimInf:Pf1} follows. To see that we can replace $\cE$ with $\overline{\cE}$, we note that \eqref{eq:GammaLimInf:Pf2} and Fatou's lemma applied to $|D[u_\veps](x,y)|^2$ implies $u \in \cS^{\rho,2}(\Omega;\bbR^n)$. Further, 
	$u_\veps = g$ almost everywhere on $\omega$ and $u_\veps \to u$ almost everywhere on $\Omega$ implies that $u = g$ almost everywhere on $\omega$, so therefore $u \in \cS^{\rho,2}_g$ and $\cE(u) = \overline{\cE}(u)$.
\end{proof}

\begin{lemma}\label{lma:Linearization:Limsup}
	For any $u \in L^2(\Omega;\bbR^n)$, 
	\begin{equation*}
		\Gammalimsup_{\veps \to 0^+} \overline{\cE}_\veps(u) \leq \overline{\cE}(u)\,.
	\end{equation*}
\end{lemma}

\begin{proof}
	The inequality is trivial if $u\in L^{2}(\Omega;\bbR^d)\setminus \cS^{\rho,2}_g$, so we assume that $u \in \cS^{\rho,2}_g$. We first assume additionally that $u \in W^{1,\infty}(\Omega;\bbR^n)$. Set $M := \Vnorm{u}_{W^{1,\infty}(\Omega)}$. Then $|\cD[u](x,y)| \leq M$ for all $(x,y) \in \Omega \times \Omega$.
	Then directly
	\begin{equation*}
		|\rme[\mathrm{id}+\veps u_\veps](x,y) - \veps D[u_\veps](x,y) \cdot D[\mathrm{id}](x,y)| \leq C(M) \veps^2.
	\end{equation*}
	Thus
	$\overline{\cE}_{\veps}(u) < \infty$ 
	for $\veps > 0$ small enough and 
	by the dominated convergence theorem
	\begin{equation*}
		\lim\limits_{\veps \to 0^+} \cE_{\veps}(u) = \cE(u).
	\end{equation*}
	For the general case $u \in \cS^{\rho,2}_g$, we use the density of smooth functions in the nonlocal space $\cS^{\rho,2}$ (the proof of which is virtually the same as that of \cite[Theorem 3.3]{Mengesha2024Linearization}), as well as the fact that $\Omega \setminus \omega$ is Lipschitz, to obtain that the closure of the function space
	\begin{equation*}
		\{ w \in W^{1,\infty}(\Omega;\bbR^n) \, : \, w = 0 \text{ on } \omega \}
	\end{equation*}
	in the $\cS^{\rho,2}(\Omega;\bbR^n)$ norm is the space
	\begin{equation*}
		\{ w \in \cS^{\rho,2}(\Omega;\bbR^n) \, : \, w = 0 \text{ a.e. on } \omega \}.
	\end{equation*}
	After ``adding $g$ to the two sets,'' we see the existence of a sequence $\{ u_k \} \subset W^{1,\infty}(\Omega;\bbR^d)$ with $u_k = g$ on $\omega$ that converges to $u \in \cS^{\rho,2}_g$ strongly in $\cS^{\rho,2}(\Omega;\bbR^n)$. Thus by the $L^2$ lower semicontinuity of the $\Gammalimsup$ and the continuity of $\cE$ with respect to strong convergence in $\cS^{\rho,2}$
	\begin{equation*}
		\Gammalimsup_{\veps \to 0^+} \overline{\cE}_\veps(u) \leq \liminf_{k \to \infty} \left( \Gammalimsup_{\veps \to 0^+} \overline{\cE}_\veps(u_k) \right) = \liminf_{k \to \infty} \cE(u_k) = \cE(u)\,.
	\end{equation*}
\end{proof}

\subsection{Existence of minimizers}\label{subsec:Existence}

\begin{theorem}\label{thm:Existence}
	Under the assumptions of \cref{thm:VariationalConvergence}, assume additionally that $W$ is lower semicontinuous on $[-1,\infty)$. 
	Then there exists $u \in \cS^{\rho,2}_{g}$ such that
	\begin{equation}\label{eq:ExistenceOfMin}
		\overline{\cF}_\veps(u) = \min_{ w \in \cS^{\rho,2}_{g} } \overline{\cF}_\veps(w).
	\end{equation}
\end{theorem}

\begin{proof}
	First, $\overline{\cF}_\veps$ is not identically $+\infty$ since $\overline{\cF}_\veps(g) < \infty$. Next, we note that $\overline{\cF}_\veps$ is bounded from below, since by \cref{thm:EnergyBound}
	\begin{equation}\label{eq:Existence:Pf1}
		\vnorm{w}_{\cS^{\rho,2}(\Omega)}^2 \leq C \cE_\veps(w) + C \vnorm{g}_{\cS^{\rho,2}(\omega)}^2 \leq C ( 1 + \cF_\veps(w) + \vnorm{w}_{L^2(\Omega)} ), \quad \forall w \in \cS^{\rho,2}_g.
	\end{equation}
	Therefore $m_\veps = \inf_{ w \in \cS^{\rho,2}_{g} } \overline{\cF}_\veps(w) \in \bbR$; let $\{ w_k \}_k \subset \cS^{\rho,2}_g$ be a minimizing sequence. The bound \eqref{eq:Existence:Pf1} implies that $\{w_k\}$ is a bounded sequence in $\cS^{\rho,2}(\Omega;\bbR^n)$. 
	Therefore we can apply the compactness result of \cite[Theorem 1.3]{Du2023$} to obtain a subsequence (not relabeled) $w_k$ that converges strongly in $L^2(\Omega;\bbR^n)$ and almost everywhere to a function $u \in \cS^{\rho,2}(\Omega;\bbR^n)$. Moreover, $w_k = g$ on $\omega$ for all $k\geq 1$, so $u \in \cS^{\rho,2}_g$.
	By the lower semicontinuity of $W$ and Fatou's lemma
	\begin{equation*}
		\cE_\veps(u) \leq \frac{1}{\veps^2} \int_\Omega \int_\Omega |x-y|^2 \rho(x-y) \liminf_{k \to \infty} W( \rme[\mathrm{id} + \veps w_k](x,y) ) \, \rmd y \, \rmd x \leq \liminf_{k \to \infty} \cE_\veps(w_k).
	\end{equation*}
	Then it follows from the $L^2$-continuity of $\scL$ that $u$ satisfies \eqref{eq:ExistenceOfMin}.
\end{proof}
It is in fact possible to prove \cref{thm:Existence} without \cref{thm:EnergyBound}; one can instead use a nonlocal Poincar\'e inequality to establish the lower bound. Such an inequality holds for nonlocal seminorms with $\rho$ satisfying \eqref{assu:Kernel} and \eqref{assu:Compact}; see for instance \cite[Corollary 1.5]{Du2023$}.

\subsection{Convergence of minimizers}

\begin{proof}[Proof of \cref{thm:VariationalConvergence}]
	By \Cref{thm:GammaConvergence} $\overline{\cF}$ is the $\Gamma$-limit of $\overline{\cF}_\veps$,
	where the $\Gamma$-limit is taken with respect to the strong topology on $L^{2}(\Omega;\bbR^n)$. Moreover, 
	the sequence of functionals $\{\overline{\cF}_\veps \}$ is equicoercive in the strong topology on $L^2(\Omega;\bbR^d)$ by \Cref{thm:Equicoercivity}. We may now apply \cite[Theorem 7.8 and Corollary 7.202]{DalMaso2012introduction} to conclude that $m_\veps \to m$, and that a sequence $\{u_\veps\}_\veps$ of ``almost minimizers'' of $\overline{\cF}_\veps$ satisfying \eqref{eq:AlmostMins} converge strongly in $L^2(\Omega;\bbR^n)$ to a minimizer $u$ of $\overline{\cF}$ with $\overline{\cF}_\veps(u_\veps) \to \overline{\cF}(u)$.
\end{proof}

\section{Further discussions on geometric rigidity}\label{sec:Further}

\subsection{Examples and Counterexamples}\label{subsec:Counterexamples}

In this section we discuss some examples of kernels $\rho$ appearing across nonlocal models in applications, and how the results of this paper can be applied to the associated energies and function spaces.

The fractional kernels (possibly with less singularity on the order of the natural logarithm)
\begin{equation*}
	\rho(z) = \frac{1}{|z|^{n+sp}}, \quad s \in (0,1), \, p \in [1,\infty), \qquad \text{ or } \qquad \rho(z) = \frac{ \mathds{1}_{B(0,1)}(z) }{|z|^{n+sp}(1-\ln(|z|))^\beta}, \quad \beta > 0,
\end{equation*}
satisfy \eqref{assu:Kernel}, \eqref{assu:GeomRigid}, and \eqref{assu:Compact}; they are not singular enough to cause the series in the proof of \cref{thm:Intro:GeometricRigidity} to diverge, but singular enough so that the associated nonlocal spaces are compactly embedded in $L^p$.
Fractional kernels with additional singularity of logarithmic order
\begin{equation*}
	\rho(z) = \frac{ \mathds{1}_{B(0,1)}(z) (-\ln(|z|))^\alpha }{|z|^{n+sp}}, \quad s \in [0,1), \, p \in [1,\infty), \, \alpha > 0,
\end{equation*}
also satisfy all three conditions (note that $s = 0$ is permitted here). In the case $\rho(z) = \rho_s(z) = |z|^{-n-sp}$, the constant $C$ in the geometric rigidity estimate \eqref{eq:GeometricRigidity} has the expression $C = \frac{C'(n,p,\Omega)}{1-s}$ for $s$ near $1$. Such singular behavior of the constant is to be expected. The quantity $(1-s)[v-\psi]_{\cS^{\rho_s,p}(\Omega)}^p \to K_{n,p} \vnorm{\grad(v-\psi)}_{L^p(\Omega)}$ as $s \to 1$, while
\begin{equation*}
	(1-s) \int_{\Omega} \int_{\Omega} \rho_s(x-y) |x-y|^p |\rme[v](x,y)|^p \, \rmd y \, \rmd x \to K_{n,p}' \int_{\Omega} 
\fint_{\bbS^{n-1}} \big| |\grad v(x) \eta| - 1 \big|^p \, \rmd \sigma(\eta) \, \rmd x
\end{equation*}
as $s \to 1$. However, $\fint_{\bbS^{n-1}} \big| |A \eta| - 1 \big|^p  \, \rmd \sigma(\eta) \approx \dist(A,\rmO(n))^p$ for all $A \in \bbR^{n \times n}$. Since estimates of the form
\begin{equation*}
	\inf_{A \in \rmO(n)} \vnorm{\grad v - A}_{L^p(\Omega)}^p \leq C \int_{\Omega} \dist(\grad v(x),\rmO(n))^p \, \rmd x
\end{equation*}
cannot hold in general for $u \in W^{1,p}$, the constant $C$ in \eqref{eq:GeometricRigidity} cannot possibly remain stable as $s \to 1$.

Kernels that belong to $L^1_{loc}(\bbR^n)$, such as
\begin{equation*}
	\rho(z) = \frac{ \mathds{1}_{B(0,\eta)(z)} }{ \eta^{n+p} }, \: \eta > 0, \qquad \rho(z) = \frac{1}{|z|^\alpha}, \: \alpha \in (0,n), \quad \text{ or } \quad \rho(z) = \rme^{-|z|^2/\beta}, \: \beta > 0,
\end{equation*}
satisfy \eqref{assu:Kernel} and \eqref{assu:GeomRigid} (and hence \cref{thm:Intro:GeometricRigidity} holds), but are not singular enough to satisfy \eqref{assu:Compact}. In fact, it is easy to construct a counterexample to the $L^p$-compact embedding result of \cite[Theorem 1.1]{Du2023$} when $p = 2$; take $\Omega$ to be the unit cube, $\Omega'$ to be the concentric cube with side length $1/2$, and $u_\veps$ to be the Fourier basis with frequency $1/\veps$ restricted to $\Omega'$. Then $\vnorm{u_\veps}_{\cS^{\rho,2}(\Omega)} \leq C \vnorm{u_\veps}_{L^2(\Omega)}$ is a bounded sequence, but $u_\veps$ does not converge strongly in $L^2$. However, it can still be shown that \cref{thm:VariationalConvergence} holds for such kernels if we consider convergence in the weak topology on $L^2$; this modification is nontrivial and will be reported in a future work.

Moreover, there exist kernels whose associated nonlocal spaces are compactly embedded in $L^2(\Omega;\bbR^n)$, but for which \cref{thm:Intro:GeometricRigidity} does not hold. Consider the kernel
\begin{equation*}
	\rho(z) = \frac{ \mathds{1}_{B(0,1)}(z) }{ |z|^{n+p} (1-\ln(|z|))^\beta }, \: \beta > 1.
\end{equation*}
Then $\rho$ satisfies \eqref{assu:Kernel} and \eqref{assu:Compact}, but not \eqref{assu:GeomRigid}; $\lim\limits_{r \to 0^+} \frac{\rho(r/2)}{\rho(r)} = 2^{n+p}$. This example thus illustrates that the condition $q < 2^{n+p}$ in \eqref{assu:GeomRigid} is sharp.

Now we show that \cref{thm:Intro:GeometricRigidity} fails for $\rho$. Precisely, for $p \in [1,\infty)$ and $\Omega = B(0,1)$ we show that there exists a family of functions $w_\tau \in W^{1,\infty}(B;\bbR^n)$ parametrized by $\tau > 0$ that satisfy
\begin{equation}\label{eq:Counterexample}
	\begin{gathered}
		\cG(w_\tau) := \int_B \int_B \rho(x-y) |x-y|^p |\rme[w_\tau](x,y)|^p \, \rmd y \, \rmd x \leq \frac{C \tau}{(1-\ln(\tau))^\beta}, \\ \text{while } [w_\tau - \psi]_{\cS^{\rho,p}(B)}^p \geq \frac{C \tau}{(1-\ln(\tau))^{\beta-1}} \quad \forall \psi \in \scR.
	\end{gathered}
\end{equation}
In fact, $w_\tau \to \mathrm{id}$ uniformly on $B$ as $\tau \to 0$, but clearly \cref{thm:Intro:GeometricRigidity} cannot hold since $\displaystyle \liminf_{\tau \to 0^+} \frac{ [w_\tau - \psi]_{\cS^{\rho,p}(B)}^p }{ \cG(w_\tau) } = + \infty$.

First consider $n = 1$ and $\Omega = (-1,1)$. Denote the Heaviside function $\mathds{1}_{[0,\infty)}(x) = H(x)$, and define for $\tau > 0$ the function $w_\tau(x) = x - 2x H(x) + 2(x-\tau)H(x-\tau)$. $w_\tau$ can be thought of as the partial self-interpenetration of the section $[0,\tau]$ via reflection of an infinite elastic bar. The displacement is $u_\tau(x) = w_\tau(x) - x = 2(x-\tau)H(x-\tau) - 2x H(x)$. 
Change variables $z = y-x$ to get
\begin{equation*}
	\cG(w_\tau) = \int_\bbR \int_{\{ |z| < 1 \} } \mathds{1}_{ (-1,1) \cap (-1-z,1-z) }(x) \frac{|\rme[w_\tau](x,x+z)|^p }{ |z| (1-\ln(|z|))^\beta }  \, \rmd z \, \rmd x.
\end{equation*}
Now, $\rme[w_\tau](x,x+z) = 0$ unless $x$ and $x+z$ are in distinct members of the collection of intervals $\{ (-1,0), (0,\tau), (\tau,1) \}$. First take $z \geq 0$. If $x < 0$ and $x+z > 0$, then $-z < x < 0$, and if $0 < x < \tau$ and $x+z > \tau$ then $\tau - z < x < \tau$. Similarly, $\rme[w_\tau](x,x+z) = 0$ for $z < 0$ implies that either $0 < x < -z$ or $\tau < x < \tau - z$.
Therefore
\begin{equation*}
	\begin{split}
		\cG(w_\tau) &\leq \int_\bbR \int_0^1 \mathds{1}_{ (-|z|,0) \cup (\tau-|z|,\tau) }(x) \frac{|\rme[w_\tau](x,x+z)|^p }{ |z| (1-\ln(|z|))^\beta }  \, \rmd z \, \rmd x \\
		&\qquad + \int_\bbR \int_{-1}^0 \mathds{1}_{ (0,|z|) \cup (\tau,\tau+|z|) }(x) \frac{|\rme[w_\tau](x,x+z)|^p }{ |z| (1-\ln(|z|))^\beta }  \, \rmd z \, \rmd x.
	\end{split}
\end{equation*}
Next, since $w_\tau$ is Lipschitz we have $||w_\tau(x)-w_\tau(y)| - |x-y|| \leq 2 |x-y|$, and also since $u_\tau$ is bounded by $2 \tau$ we have  $||w_\tau(x)-w_\tau(y)| - |x-y|| = ||x -y  + u_\tau(x)-u_\tau(y)| - |x-y|| \leq |u_\tau(x) - u_\tau(y)| \leq 4 \tau$. Thus $|\rme[w_\tau](x,y)| \leq 4 \min \{ 1, \tau/|x-y| \}$ and so by the resulting symmetry of the integrand
\begin{equation*}
	\cG(w_\tau) \leq 4^p \int_\bbR \int_0^1 \mathds{1}_{ (-z,z) \cup (\tau-z,\tau+z) }(x) \frac{ \min \{ 1, \tau/z \}^p }{ z (1-\ln(z))^\beta }  \, \rmd z \, \rmd x.
\end{equation*}
Integrating in $x$ and splitting, we therefore have
\begin{equation*}
	\cG(w_\tau) \leq C \int_0^\tau \frac{1}{(1-\ln(z))^\beta} \, \rmd z + C \int_\tau^1 \frac{\tau^p}{ z^p (1-\ln(z))^\beta } \, \rmd z \leq C \frac{\tau}{(1-\ln(\tau))^\beta}.
\end{equation*}

For the lower bound, when $n = 1$ the only $\psi$ we need to consider are $\psi(x) = x$ and $\psi(x) = -x$. When $\psi(x) = -x$, then $D[w_\tau - \psi](x,y) = 2$ for $x < 0$, $y< 0$ and therefore a lower bound for $	[w_\tau - \psi]_{\cS^{\rho,p}(\Omega)}^p$ holds independent of $\tau$. When $\psi(x) = x$ then we use the portion of the seminorm on the reflected part of the reference configuration to obtain
\begin{equation*}
	[w_\tau - \psi]_{\cS^{\rho,p}(\Omega)}^p \geq \int_{\tau/4}^{3\tau/4} \int_{x-\tau/8}^{x + \tau/8} \rho(x-y) 2^p|x-y|^p \, \rmd y \, \rmd x = C \tau \int_0^{\tau/8} \frac{1}{z(1-\ln(z))^\beta} \, \rmd z
\end{equation*}
A change of variables in the last integral gives the lower bound in \eqref{eq:Counterexample}.

For general $n \geq 2$, define $v_\tau(x) = (w_\tau(x_1),x_2,\ldots,x_n)$. Change variables $z = y-x$; similarly to the above, for all $|z| \leq 1$ the $\cL^n$-measure of $\{ x \in B \, : \, \rme[v_\tau](x,x+z) \neq 0 \}$ does not exceed $C(n) |z|$. The bound $|\rme[w_\tau](x,y)| \leq 4 \min \{ 1, \tau/|x-y| \}$ still holds, and using polar coordinates we arrive at the analogous upper bound on $\cG(w_\tau)$.

For the lower bound, we need to consider general $\psi \in \scR$, more precisely general $A \in \rmO(n)$. 
By restricting the domain of integration and then converting to polar coordinates we have
\begin{equation*}
	\begin{split}
	[v_\tau - A(\cdot)]_{\cS^{\rho,p}(\Omega)}^p &\geq \int_{B \cap \{ -3/4 < x_1 < -1/4 \} } \int_{B(x, 1/8)} \rho(x-y) |(I-A)(x-y)|^p \, \rmd y \, \rmd x \\ 
	&\quad + \int_{B(0,1/4) \cap \{ \tau/4 < x_1 < 3\tau/4 \} } \int_{B(x, \tau/8)} \rho(x-y) |(B-A)(x-y)|^p \, \rmd y \, \rmd x \\
	&\geq C \int_0^{1/8} \frac{1}{r (1-\ln(r))^\beta } \, \rmd r \int_{\bbS^{n-1}} |(I-A)\eta|^p \, \rmd \sigma(\eta) \\
	&\quad + C \tau \int_0^{\tau/8} \frac{1}{r (1-\ln(r))^\beta } \, \rmd r \int_{\bbS^{n-1}} |(B-A)\eta|^p \, \rmd \sigma(\eta), 
	\end{split}
\end{equation*}
where $B = \diag(-1,1,\ldots,1)$ (Note that $(\int_{\bbS^{n-1}} |(I-A)\eta|^p \, \rmd \sigma(\eta))^{1/p}$ defines a norm on matrices).
If $|A - I| > 1$ then we use the first term to obtain a lower bound independent of $\tau$ and $A$. If $|A-I| \leq 1$ then $|B-A| \geq 1$, and the second term gives the lower bound $\frac{C\tau}{(1-\ln(\tau))^{\beta-1}}$ independent of $A$. Thus \eqref{eq:Counterexample} is established for $n \geq 2$.

\subsection{Nonlocal Korn's inequalities}

\begin{proof}[Proof of \cref{thm:Korn}]
	Thanks to the density result of \cite[Theorem 3.3]{Mengesha2024Linearization} it suffices to prove \eqref{eq:Korn} for $u \in C^\infty(\overline{\Omega})$. To choose the vector field $\psi$, we use a projection operator from \cite{MengeshaConsistency} inspired by \cite{Kohn1979New}. Let $\bx_0 \in \Omega$ and abbreviate $B = B(x_0,\frac{1}{2} \dist(x_0,\p \Omega))$. Denote the class of skew-symmetric matrices as $\mathrm{Skw}(n) = \{ A \in \bbR^{n \times n} \, : \, A = - A^T \}$, and denote the class of \textit{infinitesimal rigid motions} as $\scX = \{ Ax + b \, : A \in \mathrm{Skw}(n), b \in \bbR^n \}$. Define $\Pi : L^1(B;\bbR^n) \to \scX$ to be the projection operator as introduced in \cite[Lemma 2.1]{MengeshaConsistency}. Then $\Pi$ is a continuous linear operator with $\vnorm{\Pi[u]}_{\mathrm{aff}} \leq C(n,p) \vnorm{u}_{L^p(B)}$ for all $p \in [1,\infty]$, with $\Pi[A(\cdot) + b](x) = \frac{1}{2}(A-A^T)x + b$ for any affine vector field. We choose $\psi = \Pi[u]$.
	
	We may replace $u$ with $u - \Pi[u]$, since the projected difference seminorm on the right-hand side of \eqref{eq:Korn} is unchanged by subtracting infinitesimal rigid motions. Therefore we assume henceforth that $\psi = \Pi[u] = 0$.
	
	For $\veps > 0$ sufficiently small, apply \cref{thm:Intro:GeometricRigidity} and \cref{thm:GeometricRigidity:Poincare} to the function $v_\veps(x) = x+ \veps u(x)$ to obtain $\psi_\veps = A_\veps x + b_\veps \in \scR$ such that
	\begin{equation}\label{eq:Korn:Pf1}
		\vnorm{v_\veps - \psi_\veps}_{\cS^{\rho,p}(\Omega)}^p \leq C \int_{\Omega} \int_{\Omega} \rho(x-y) |x-y|^p |\rme[v_\veps](x,y)|^p \, \rmd y \, \rmd x.
	\end{equation}
	Similarly to \eqref{eq:LinearizedStrainPotential}, 
	\begin{equation*}
		|\rme[v_\veps](x,y)|^p = C(p) \veps^p |D[u](x,y) \cdot D[\mathrm{id}](x,y)|^p + O(\veps^{2p} |x-y|^p)
	\end{equation*}
	for all $x$ and $y$ and all $\veps > 0$. Multiplying by $\rho(x-y)|x-y|^p$ and integrating, the error term remains bounded by $\veps^{2p}$ after integration, and so the right-hand side of \eqref{eq:Korn:Pf1} becomes 
	\begin{equation*}
		C \veps^p \int_\Omega \int_\Omega \rho(x-y) |x-y|^p \big| D[u](x,y) \cdot D[\mathrm{id}](x,y)\big|^p \, \rmd y \, \rmd x + O(\veps^{2p}) := C \veps^p [u]_{\cX^{\rho,p}(\Omega)}^p + O(\veps^{2p}).
	\end{equation*}

	We need to find the leading order of the left-hand side of \eqref{eq:Korn:Pf1}, which requires us to get a better understanding of the rigid motion $\psi_\veps$. Precisely, we need to estimate the distance of a matrix from $I$ using its projection to the tangent space of $\rmO(n)$ at $I$.
	First,
	\begin{equation}\label{eq:Korn:Pf2}
		\begin{split}
			\veps^p \vnorm{u}_{\cS^{\rho,p}(\Omega)}^p &\leq C \vnorm{ \mathrm{id} + \veps u - \psi_\veps}_{\cS^{\rho,p}(\Omega)}^p + C  \vnorm{ \mathrm{id} - \psi_\veps}_{\cS^{\rho,p}(\Omega)}^p \\
			&\leq C \vnorm{ \mathrm{id} - \psi_\veps}_{\mathrm{aff}}^p + C \veps^p [u]_{\cX^{\rho,p}(\Omega)}^p + O(\veps^{2p}).
		\end{split}
	\end{equation}
	Then since $u \in C^\infty$, $\veps^p \vnorm{u}_{\cS^{\rho,p}(\Omega)}^p \to 0$ as $\veps \to 0$, and the same inequality \eqref{eq:Korn:Pf2} with the roles of $\veps u$ and $\mathrm{id} - \psi_\veps$ reversed shows that $ \vnorm{ \mathrm{id} - \psi_\veps}_{\mathrm{aff}} \to 0$ as $\veps \to 0$. Let $\veps_0 > 0$ be such that $ \vnorm{ \mathrm{id} - \psi_\veps}_{\mathrm{aff}} \leq 1$ for all $\veps < \veps_0$. Then
	\begin{equation*}
		|A_\veps - I| \leq 2|A_\veps - I| - |A_\veps-I|^2 \leq | 2(A_\veps - I) + (A_\veps - I)^2 | = |A_\veps^2 - I| = |A_\veps - A_\veps^T|,
	\end{equation*}
	where in the last equality we used that $A_\veps$ belongs to $\rmO(n)$. Now, by continuity of the projection $\Pi$, and since $\Pi[v_\veps] = \Pi[\mathrm{id} + \veps u] = 0$, we have for $\veps < \veps_0$
	\begin{equation*}
		\begin{split}
			\vnorm{ \mathrm{id} - \psi_\veps}_{\mathrm{aff}} = |A_\veps - I| + |b_\veps| &\leq |A_\veps - A_\veps^T| + |b_\veps| \\
			&= \vnorm{ \Pi[\psi_\veps] }_{\mathrm{aff}}
			= \vnorm{ \Pi[v_\veps - \psi_\veps] }_{\mathrm{aff}} \\
			&\leq \vnorm{v_\veps - \psi_\veps}_{\cS^{\rho,p}(\Omega)}
			\leq C \veps [u]_{\cX^{\rho,p}(\Omega)} + O(\veps^{2}).
		\end{split}
	\end{equation*}
	Therefore using \eqref{eq:Korn:Pf2} and dividing by $\veps^p$,
	\begin{equation*}
		\vnorm{u}_{\cS^{\rho,p}(\Omega)}^p \leq C \int_\Omega \int_\Omega \rho(x-y) |x-y|^p \big| D[u](x,y) \cdot D[\mathrm{id}](x,y)\big|^p \, \rmd y \, \rmd x + O(\veps^{p}),
	\end{equation*}
	and the result follows by taking $\veps \to 0$.
\end{proof}

\bibliography{ReferencesFull}
\bibliographystyle{plain}

\end{document}